\documentclass{birkjour}

\usepackage{amsmath}
\usepackage{amssymb}
\usepackage{verbatim}   
\usepackage{color}    
\usepackage{mathrsfs,epsfig}  
\usepackage{fontenc}
\usepackage{textcomp}
\usepackage{fix-cm}
\usepackage{array}
\usepackage{enumerate}
\usepackage{latexsym, amsfonts, amssymb}
\usepackage{color}
\usepackage{comment}
\usepackage{subfigure}
\usepackage{accents}

\numberwithin{equation}{section}

\newtheorem{theorem}[equation]{Theorem}
\newtheorem{proposition}[equation]{Proposition}
\newtheorem{lemma}[equation]{Lemma}
\newtheorem{corollary}[equation]{Corollary}

\theoremstyle{definition}

\newtheorem{remark}[equation]{Remark}

\begin{document}
\fontsize{9}{11}\selectfont

\title[Generalized Inviscid Proudman-Johnson Equation]
{Blow-up of Solutions to the Generalized\\ Inviscid Proudman-Johnson Equation}

\author{Alejandro Sarria}
\address{%
Department of Mathematics\\
University of New Orleans\\
New Orleans, LA, 70148, USA}
\email{asarria1@uno.edu}

\author{Ralph Saxton}
\address{%
Department of Mathematics\\
University of New Orleans\\
New Orleans, LA, 70148, USA
}
\email{rsaxton@uno.edu}



\begin{abstract}
For arbitrary values of a parameter $\lambda\in\mathbb{R},$ finite-time blow-up of solutions to the generalized, inviscid Proudman-Johnson equation is studied via a direct approach which involves the derivation of representation formulae for  solutions to the problem.
\end{abstract}

\maketitle

\section{Introduction}
\label{sec:intro}

In this article, we examine  blow-up, and blow-up properties, in solutions to the initial boundary value problem 
\begin{equation}
\label{eq:nonhomo}
\begin{cases}
u_{xt}+uu_{xx}-\lambda u_x^2=I(t),\,\,\,\,\,&\,t>0,
\\
u(x,0)=u_0(x),\,\,\,\,&x\in[0,1],
\\
I(t)=-(\lambda+1)\int_0^1{u_x^2\,dx,}
\end{cases}
\end{equation}
where $\lambda\in\mathbb{R}$, and solutions are subject to  periodic boundary conditions
\begin{equation}
\label{eq:pbc}
u(0,t)=u(1,t),\,\,\,\,\,\,\,u_x(0,t)=u_x(1,t).
\end{equation}
Equations (\ref{eq:nonhomo})i), iii) may be obtained by integrating the partial differential equation
\begin{equation}
\label{eq:ipj}
u_{xxt}+uu_{xxx}+(1-2\lambda)u_xu_{xx}=0
\end{equation}
and using (\ref{eq:pbc}) (\cite{Proudman1}, \cite{Childress}, \cite{Okamoto1})\footnote[1]{Equation (\ref{eq:ipj}) was introduced in \cite{Okamoto1} with a parameter $a\in\mathbb{R}$ instead of the term $2\lambda-1$. Since in this article we will be concerned with equation (\ref{eq:nonhomo})i), iii), our choice of the parameter $\lambda$ over $a$ is due, mostly, to notational convenience.}. We refer to (\ref{eq:nonhomo}) as  the generalized, inviscid, Proudman-Johnson equation and note that the equation occurs in several different contexts, either with or without the nonlocal term $I(t)$. For $\lambda=-1,$ it reduces to Burgers' equation. If $\lambda=-1/2,$ the Hunter Saxton (HS) equation  describes the orientation of waves in massive director nematic liquid crystals (\cite{Hunter1}, \cite{Bressan1}, \cite{Dafermos1}, \cite{Yin1}). For periodic functions, the HS-equation also describes geodesics on the group $\mathcal{D}(\mathbb{S})\backslash Rot(\mathbb{S})$ of orientation preserving diffeomorphisms on the unit circle $\mathbb{S}=\mathbb{R}\backslash\mathbb{Z}$, modulo the subgroup of rigid rotations with respect to the right-invariant metric $\left\langle f,g\right\rangle=\int_{\mathbb{S}}{f_xg_xdx}$ (\cite{Khesin1}, \cite{Bressan1}, \cite{Tiglay1}, \cite{Lenells1}). If $\lambda=\frac{1}{n-1},\,n\geq2,$ (\ref{eq:nonhomo}) i), iii) can be obtained directly from the $n-$dimensional incompressible Euler equations $$\boldsymbol{u}_t+(\boldsymbol{u}\cdot\nabla)\boldsymbol{u}=-\nabla p,\,\,\,\,\,\,\,\,\nabla\cdot \boldsymbol{u}=0$$ using  stagnation point form velocities $\boldsymbol{u}(x,\boldsymbol{x}^\prime,t)=(u(x,t),-\lambda\boldsymbol{x}^\prime u_x(x,t))$ for $\boldsymbol{x}^\prime=\{x_2,...,x_n\},$ or through the  cylindrical coordinate representation  $u^r=-\lambda ru_x(x,t),$ $u^{\theta}=0$ and $u^x=u(x,t)$, where $r=\left|\boldsymbol{x}^\prime\right|,$ (\cite{Childress}, \cite{Weyl1}, \cite{Saxton1}, \cite{Okamoto1}, \cite{Escher1}). Finally, in the local case $I(t)=0$, the equation appears as a special case of Calogero's equation $$u_{xt}+uu_{xx}-\Phi(u_x)=0$$ for arbitrary functions $\Phi(\cdot)$ (\cite{Calogero1}). The earliest results on blow-up in the nonlocal case $I(t)=-2\int_0^1{u_x^2dx}$ for $\lambda=1$ are due to Childress et al. (\cite{Childress}), where the authors show that there are blow-up solutions under Dirichlet boundary conditions. For spatially periodic solutions, the following is known:
\begin{itemize}
\item If $\lambda\in[-1/2,0)$ and $u_0(x)\in W_{\mathbb{R}}^{1,2}(0,1)$, $\left\|u_x\right\|_2$ remains bounded but $\left\|u_x\right\|_{\infty}$ blows up (\cite{Okamoto2}). For $\lambda\in[-1,0),$ if $u_0(x)\in H_{\mathbb{R}}^s(0,1),\,s\geq3$ and $u_0^{\prime\prime}$ is not constant, $\left\|u_x\right\|_{\infty}$ blows up (\cite{Wunsch1}), similarly if $\lambda\in(-2,-1)$, as long as 
\begin{equation}
\label{eq:wunschcond}
\inf_{x\in[0,1]}\left\{u_0^{\prime}(x)\right\}+\sup_{x\in[0,1]}\left\{u_0^{\prime}(x)\right\}<0.
\end{equation}
\item For $\lambda\in(-\infty,-1/2),\,\left\|u_x\right\|_2$ blows up in finite-time as long as (\cite{Okamoto1}) 
\begin{equation}
\label{eq:okamotocond}
\int_0^1{u_0^\prime(x)^3dx}<0.
\end{equation}
\item If $\lambda\in[0,1/2)$ and $u_0^{\prime\prime}(x)\in L_{\mathbb{R}}^{\frac{1}{1-2\lambda}}(0,1),\,u$ exists globally in time. Similarly, for $\lambda=1/2$ as long as $u_0(x)\in W_{\mathbb{R}}^{2,\infty}(0,1)$ (\cite{Okamoto2}, \cite{Saxton1}).
\item If $\lambda\in[1/2,1)$ and $u_0'''(x)\in L_{\mathbb{R}}^{\frac{1}{2(1-\lambda)}}(0,1), u$ exists globally in time (\cite{Okamoto2}). 
\end{itemize}
The purpose of this paper is to provide further insight on how periodic solutions to (\ref{eq:nonhomo}) blow up for parameters $\lambda\in(-\infty,0)$ as well as to study  regularity under differing assumptions on initial data when $\lambda\in[0,+\infty).$ To do this, we will examine solutions arising out of several classes of periodic, mean zero, initial data: the first, a class of smooth functions $u_0(x)\in C_\mathbb{R}^\infty(0,1)$, and then two classes of  data for which either $u_0'(x)$ or $u_0'^\prime(x)\in PC_\mathbb{R}(0,1),$ the family of piecewise constant functions.  The results are obtained via a direct approach which will involve the derivation of representation formulae for $u_x$ along characteristics. The rest of the paper is organized as follows.  A brief summary of new blow-up results is given in \S \ref{sec:summary}. The derivation of the solution representation formulae and proofs of the results are given in \S \ref{sec:sol} and \S \ref{sec:blow}, respectively, as well as in appendix A. Finally, some illustrative examples are to be found in \S \ref{sec:examples}.

\section{Summary of Results}
\label{sec:summary}
Our first aim will be to obtain the representation formula, (\ref{eq:mainsolu}), for solutions to (\ref{eq:nonhomo})-(\ref{eq:pbc}), which will permit us to estimate their lifetime for arbitrary  $\lambda\in\mathbb{R}$. Given  $\eta_*\in\mathbb{R}^+$, to be defined, blow-up of solutions will depend upon the existence of a finite, positive, limit $t_*$ defined by
\begin{equation}
\label{eq:assympttwo}
\begin{split}
t_*\equiv\lim_{\eta\uparrow\eta_*}{\int_0^{\eta}{\left(\int_0^1{\frac{d\alpha}{(1-\lambda\mu u_0^\prime(\alpha))^{\frac{1}{\lambda}}}}\right)^{2\lambda}\,d\mu}}.
\end{split}
\end{equation}
Let us suppose  a solution $u(x,t)$ exists on an interval $t\in[0, T],\, T<t_*.$ 
Denote by  $\gamma(\alpha,t)$  the solution to the initial value problem
\begin{equation}
\label{eq:cha}
\dot\gamma(\alpha,t)=u(\gamma(\alpha,t),t),\,\,\,\,\,\,\,\,\,\,\,\,\,\,\,\,\,\gamma(\alpha,0)=\alpha\in[0,1],
\end{equation}
and define
\begin{equation}
\label{eq:max}
\begin{split}
M(t)\equiv\sup_{\alpha\in[0,1]}\{u_x(\gamma(\alpha,t),t)\},\,\,\,\,\,\,\,\,\,\,M(0)=M_0,
\end{split}
\end{equation}
and
\begin{equation}
\label{eq:min22}
\begin{split}
m(t)\equiv\inf_{\alpha\in[0,1]}\{u_x(\gamma(\alpha,t),t)\},\,\,\,\,\,\,\,\,\,\,m(0)=m_0,
\end{split}
\end{equation}
where, for $u_0(\alpha)\in C_{\mathbb{R}}^\infty(0,1)$ and $\lambda>0,$ we will assume that the mean-zero function $u_0^\prime$ attains its greatest value $M_0>0$ at, at most, finitely many locations $\overline\alpha_i\in[0,1],\,\,1\leq i\leq m.$ Similarly, for $\lambda<0,$ we suppose that the  least value, $m_0<0$, occurs at a discrete set of points\footnote[2]{One possibility for admitting infinitely many $\overline\alpha_i$ and/or $\underline\alpha_j$ will be considered below for the class $PC_{\mathbb{R}}(0,1).$} $\underline\alpha_j\in[0,1],\,\,1\leq j\leq n.$ From the above definitions and the solution formula, it can easily be shown that (see appendix C) 
\begin{equation}
\label{eq:maxp}
\begin{split}
M(t)=u_x(\gamma(\overline\alpha_i,t),t),\,\,\,\,\,\,\,\,\,\,\,m(t)=u_x(\gamma(\underline\alpha_j,t),t).
\end{split}
\end{equation}
The main results of this paper are summarized in the following theorems and in Corollary \ref{coro:coro1} below. 
\begin{theorem}
\label{thm:maintheorem1}
Consider the initial boundary value problem (\ref{eq:nonhomo})-(\ref{eq:pbc}) for the generalized, inviscid, Proudman-Johnson equation. There exist smooth, mean-zero initial data such that:
	\begin{enumerate}
		\item\label{item:global} For $\lambda\in[0,1],$ solutions exist globally in time. Particularly, these vanish as $t\uparrow t_*=+\infty$ for $\lambda\in(0,1)$ but converge to a non-trivial steady-state if $\lambda=1.$
	\item\label{item:blowtwo} For $\lambda\in\mathbb{R}\backslash(-2,1],$ there exists a finite $t_*>0$ such that both the maximum $M(t)$ and the minimum $m(t)$ diverge to $+\infty$ and  to $-\infty$, respectively, as $t\uparrow t_*.$ In addition, for every $\alpha\notin\{\overline\alpha_i,\underline\alpha_j\},$ $\lim_{t\uparrow t_*}\left|u_x(\gamma(\alpha,t),t)\right|=+\infty$ (two-sided, everywhere blow-up).
	\item\label{item:blowone} For $\lambda\in(-2,0),$ there is a finite $t_*>0$ such that only the minimum diverges, $m(t)\to-\infty,$ as $t\uparrow t_*$ (one-sided, discrete blow-up). 
	\end{enumerate}
\end{theorem}
Subsequent  results examine the behaviour, as $t\uparrow t_*$, of two quantities, the jacobian $\gamma_\alpha(\alpha,t)$ (see (\ref{eq:cha})), and the $L^p$ norm
\begin{equation}
\label{eq:energygen}
\begin{split}
\left\|u_x(x,t)\right\|_p=\left(\int_0^1{(u_x(\gamma(\alpha,t),t))^p\gamma_\alpha(\alpha,t)\,d\alpha}\right)^{1/p},\,\,\,\,\, p\in[1,+\infty),
\end{split}
\end{equation}
with particular emphasis given to the energy function $E(t)=\left\|u_x\right\|_2^2$. 
\begin{remark}
Corollary \ref{coro:coro1} and Theorem \ref{thm:lpintro} below describe  pointwise behaviour and $L^p$-regularity of solutions as $t\uparrow t_*$ where, for $\lambda\in\mathbb{R}\backslash[0,1],$ $t_*>0$ refers to the \textit{finite} $L^{\infty}$ blow-up time of Theorem \ref{thm:maintheorem1}; otherwise the description is asymptotic, for $t\uparrow t_*=+\infty.$
\end{remark}
\begin{corollary}
\label{coro:coro1}
Let $u(x,t)$ in Theorem \ref{thm:maintheorem1} be a solution to the initial boundary value problem (\ref{eq:nonhomo})-(\ref{eq:pbc}) defined for $t\in[0,t_*)$. Then
\begin{equation}
\label{eq:sumjacblow}
\lim_{t\uparrow t_*}\gamma_\alpha(\alpha,t)=
\begin{cases}
+\infty,\,\,\,\,\,\,\,\,\,\,\,\,\,&\alpha=\overline\alpha_i,\,\,\,\,\,\,\,\,\,\,\,\,\,\,\,\,\,\,\,\,\,\,\lambda\in(0,+\infty),
\\
0,\,\,\,\,\,\,\,\,\,\,\,\,\,&\alpha\neq\overline\alpha_i,\,\,\,\,\,\,\,\,\,\,\,\,\,\,\,\,\,\,\,\,\,\,\lambda\in(0,2],
\\
C,\,\,\,\,\,\,\,\,\,\,\,\,\,&\alpha\neq\overline\alpha_i,\,\,\,\,\,\,\,\,\,\,\,\,\,\,\,\,\,\,\,\,\,\lambda\in(2,+\infty),
\\
0,\,\,\,\,\,\,\,\,\,\,\,\,\,&\alpha=\underline\alpha_j,\,\,\,\,\,\,\,\,\,\,\,\,\,\,\,\,\,\,\,\,\lambda\in(-\infty,0),
\\
C,\,\,\,\,\,\,\,\,\,\,\,\,\,&\alpha\neq\underline\alpha_j,\,\,\,\,\,\,\,\,\,\,\,\,\,\,\,\,\,\,\,\,\lambda\in(-\infty,0)
\end{cases}
\end{equation}
for positive constants $C$ which depend on the choice of $\lambda$ and $\alpha$.
\end{corollary}

\begin{table}[t]
\caption{Energy Estimates and $L^p$ Regularity as $t\uparrow t_*$}
\centering
		\renewcommand{\arraystretch}{1.4}
    \begin{tabular}{ | m{3.0cm} | m{1.4cm} | m{1.4cm} | m{2.0cm} | m{2.0cm} | }
    \hline
    \centering{$\lambda$} & $\,\,\,\,\,\,E(t)$  & $\,\,\,\,\,\,\dot E(t)$ & $\,\,\,\,\,\,\,\,\,\,\,\,u_x$ \\ \hline
    \centering{$(-\infty,-2]$} &$\,\,\,\,\,+\infty$ & $\,\,\,\,\,+\infty$ & $\,\,\,\notin L^p,\, p>1$ \\ \hline
    \centering{$(-2,-2/3]$} &$\,\,\,\,\,+\infty$ & $\,\,\,\,\,+\infty$ & $\,\,\,\in L^1,\, \notin L^2$    \\ \hline
    \centering{$(-2/3,-1/2)$} &Bounded &$\,\,\,\,\,+\infty$ & $\,\,\,\in L^2,\, \notin L^3$ \\ \hline
    \centering{$-1/2$} &Constant &$\,\,\,\,\,\,\,\,\,\,\,0$ & $\,\,\,\in L^2,\, \notin L^3$\\ \hline
    \centering{$(-1/2,-2/5]$} &Bounded &$\,\,\,\,\,-\infty$ & $\,\,\,\in L^2,\, \notin L^3$ \\ \hline
    \centering{$\left(-\frac{2}{2p-1},0\right),\, p\geq3$} &Bounded &Bounded & $\,\,\,\,\,\,\,\,\,\,\,\in L^p$  \\ \hline
    \centering{$\left[-\frac{2}{p-1},-\frac{2}{p}\right],\, p\geq6$} &Bounded &Bounded & $\,\,\,\,\,\,\,\,\,\,\,\notin L^p$  \\ \hline
    \centering{$[0,1]$} &Bounded &Bounded & $\,\,\,\,\,\,\,\,\,\,\,\in L^{\infty}$     \\ \hline
    \centering{$(1,+\infty)$} &$\,\,\,\,\,+\infty$ & $\,\,\,\,\,+\infty$ & $\,\,\,\notin L^p,\, p>1$ \\ \hline
    \end{tabular}
\label{table:lptable}
\end{table}

\begin{theorem}
\label{thm:lpintro}
Let $u(x,t)$ in Theorem \ref{thm:maintheorem1} be a solution to the initial boundary value problem (\ref{eq:nonhomo})-(\ref{eq:pbc}) defined for $t\in[0,t_*)$. It holds,
\begin{enumerate}
\item\label{it:lp2} For $p\geq1$ and $\frac{2}{1-2p}<\lambda\leq1,\, \lim_{t\uparrow t_*}\left\|u_x\right\|_p<+\infty.$
\item\label{it:lp1} For $p>1$ and $\lambda\in\mathbb{R}\backslash(-2,1],\, \lim_{t\uparrow t_*}\left\|u_x\right\|_p=+\infty.$ Similarly, for $p\in(1,+\infty)$ and   $\lambda\in(-2,-2/p].$
\item\label{it:ener} The energy $E(t)=\left\|u_x\right\|_2^2$ diverges if $\lambda\in\mathbb{R}\backslash(-2/3,1]$ as $t\uparrow t_*$ but remains finite for $t\in[0,t_*]$ otherwise. Moreover, $\dot E(t)$ blows up to $+\infty$ as $t\uparrow t_*$ when $\lambda\in\mathbb{R}\backslash[-1/2,1]$ and $\dot E(t)\equiv0$ for $\lambda=-1/2;$ whereas, $\lim_{t\uparrow t_*}\dot E(t)=-\infty$ if $\lambda\in(-1/2,-2/5]$ but remains bounded when $\lambda\in(-2/5,1]$ for all $t\in[0,t_*]$. 
\end{enumerate} 
\end{theorem}
See Table \ref{table:lptable} for a summary of the results mentioned in Theorem \ref{thm:lpintro}. 
\begin{remark}
Global weak solutions to (\ref{eq:nonhomo})i) having $I(t)=0$ and $\lambda=-1/2$ have been studied by several authors, (\cite{Hunter2}, \cite{Bressan1}, \cite{Lenells1}). Such solutions have also been constructed for $\lambda\in[-1/2,0)$ in \cite{Wunsch2} (c.f. also \cite{Wunsch3}) by extending an argument used in \cite{Bressan1}. Notice that theorems \ref{thm:maintheorem1} and \ref{thm:lpintro} above imply the existence of smooth data and a finite $t_*>0$ such that strong solutions to (\ref{eq:nonhomo})-(\ref{eq:pbc}) with $\lambda\in(-2/3,0)$ satisfy $\lim_{t\uparrow t_*}\left\|u_x\right\|_{\infty}=+\infty$ but $\lim_{t\uparrow t_*}E(t)<+\infty.$ As a result, it is possible that the representation formulae derived in \S \ref{sec:sol} can lead to similar construction of global, weak solutions for $\lambda\in(-2/3,0).$
\end{remark}
The results stated thus far will be established for a family of smooth functions $u_0(x)\in C_{\mathbb{R}}^{\infty}(0,1)$ having, relative to the sign of $\lambda$, global extrema attained at finitely many points. If we next consider periodic $u_0^\prime(x)\in PC_{\mathbb{R}}(0,1),$ the class of mean-zero, piecewise constant functions, the following  holds instead:
\begin{theorem}
\label{thm:pc0}
For the initial boundary value problem (\ref{eq:nonhomo})-(\ref{eq:pbc}) with $u_0^\prime(\alpha)\in PC_{\mathbb{R}}(0,1)$ assume solutions are defined for all $t\in[0,T],\, T>0.$ Then no  $W^{1,\infty}(0,1)$ solution may exist for $T\geq t_*$, where $0<t_*<+\infty$ if $\lambda\in(-\infty,0)$, and $t_*=+\infty$ for $\lambda\in[0,+\infty)$. Further, $\lim_{t\uparrow t_*}\left\|u_x\right\|_1=+\infty$ when $\lambda\in(-\infty,-1)$ while 
\begin{equation*}
\lim_{t\uparrow t_*}\left\|u_x\right\|_p=
\begin{cases}
C,\,\,\,\,\,\,\,\,&-\frac{1}{p}\leq\lambda<0,
\\
+\infty,\,\,\,\,\,\,\,\,&-1\leq\lambda<-\frac{1}{p},
\end{cases}
\end{equation*}
for $p\geq1$ and where the positive constants $C$ depend on the choice of $\lambda$ and $p.$
\end{theorem}
Finally, the case of periodic $u_0^{\prime\prime}\in PC_{\mathbb{R}}$ is briefly examined in \S \ref{subsubsec:pl} via a simple example. Our findings are summarized in Theorem \ref{thm:pq} below.
\begin{theorem}
\label{thm:pq}
For the initial boundary value problem (\ref{eq:nonhomo})-(\ref{eq:pbc}) with $u_0^{\prime\prime}(\alpha)\in PC_{\mathbb{R}}(0,1)$ and $\lambda\in\mathbb{R}\backslash[0,1/2],$ there are blow-up solutions. Specifically, when $\lambda\in(1/2,+\infty),$ solutions can undergo a two-sided, everywhere blow-up in finite-time, whereas for $\lambda\in(-\infty,0),$ divergence of the minimum to negative infinity can occur at a finite number of locations.
\end{theorem}
\begin{remark}
In addition to providing an approach for the case $\lambda\in(1,+\infty)$ and giving  a more detailed description of the $L^p$ regularity of solutions, the advantage of having the solution formula  (\ref{eq:mainsolu}) available is that conditions such as (\ref{eq:wunschcond}) and (\ref{eq:okamotocond}), though sufficient for blow-up, will not be necessary in our future arguments. 
\end{remark}

\section{The General Solution}
\label{sec:sol}
We now establish our solution formulae for (\ref{eq:nonhomo})-(\ref{eq:pbc}). Given $\lambda\in\mathbb{R}\backslash\{0\},$ equations (\ref{eq:nonhomo})i), iii) admit a second-order, linear, ordinary differential equation for the jacobian $\gamma_{\alpha}(\alpha,t).$ The case $\lambda=0$ will be considered separately in appendix A. In the reformulated problem, a general solution  is constructed
which shows $u_x(\gamma(\alpha,t),t)$ to satisfy (\ref{eq:nonhomo})i) along characteristics, namely
\begin{equation}
\label{eq:ref}
\frac{d}{dt}(u_x(\gamma(\alpha,t),t))-\lambda u_x(\gamma(\alpha,t),t)^2=I(t).
\end{equation}
Since $\dot\gamma(\alpha,t)=u(\gamma(\alpha,t),t),$
\begin{equation}
\label{eq:jacid}
\begin{split}
\dot\gamma_{\alpha}=(u_x\circ\gamma)\cdot\gamma_{\alpha}
\end{split}
\end{equation}
therefore, using (\ref{eq:nonhomo}) and (\ref{eq:jacid}),
\begin{equation}
\label{eq:chain}
\begin{split}
\ddot\gamma_\alpha&=(u_{xt}+uu_{xx})\circ\gamma\cdot\gamma_\alpha+(u_x\circ\gamma)\cdot\dot\gamma_\alpha
\\
&=(u_{xt}+uu_{xx})\circ\gamma\cdot\gamma_\alpha+u_x^2\circ\gamma\cdot\gamma_\alpha
\\
&=(\lambda+1)\left(u_x^2\circ\gamma-\int_0^1{u_x^2dx}\right)\cdot\gamma_\alpha
\\
&=(\lambda+1)\left((\gamma^{-1}_{\alpha}\cdot\dot\gamma_{\alpha})^2-\int_0^1{u_x^2dx}\right)\cdot\gamma_\alpha\,.
\end{split}
\end{equation}
For $I(t)=-(\lambda+1)\int_0^1{u_x^2dx}$ and $\lambda\in\mathbb{R}\backslash\{0\},$ then
\begin{equation}
\label{eq:jacid2}
\begin{split}
I(t)=\frac{\ddot\gamma_{\alpha}\cdot\gamma_{\alpha}-(\lambda+1)\cdot\dot\gamma_{\alpha}^{\,2}}{\gamma_{\alpha}^{\,2}}=-\frac{\gamma_{\alpha}^{\,\lambda}\cdot(\gamma_{\alpha}^{-\lambda})^{\ddot{}}}{\lambda}
\end{split}
\end{equation}
and so
\begin{equation}
\label{eq:jacprob}
\begin{split}
(\gamma_{\alpha}^{-\lambda})^{\ddot{}}+\lambda\gamma_{\alpha}^{-\lambda}\,I(t)=0.
\end{split}
\end{equation}
Setting
\begin{equation}
\label{eq:lin}
\begin{split}
\omega(\alpha,t)=\gamma_{\alpha}(\alpha,t)^{-\lambda}
\end{split}
\end{equation}
yields
\begin{equation}
\label{eq:nonhomo2}
\begin{split}
\ddot\omega(\alpha,t)+\lambda I(t)\omega(\alpha,t)=0,
\end{split}
\end{equation}
an ordinary differential equation parametrized by $\alpha$. Suppose we have two linearly independent solutions $\phi_1(t)$ and $\phi_2(t)$ to (\ref{eq:nonhomo2}), satisfying $\phi_1(0)=\dot{\phi}_2(0)=1$, $\dot{\phi}_1(0)=\phi_2(0)=0$. Then by Abel's formula,  $\text{W}(\phi_1(t),\phi_2(t))=1,\, t\geq 0$, where W$(g,h)$ denotes the wronskian of $g$ and $h.$ We look for solutions of (\ref{eq:nonhomo2}), satisfying appropriate initial data, of the form
\begin{equation}
\label{eq:unonhomosol2}
\begin{split}
\omega(\alpha,t)=c_1(\alpha)\phi_1(t)+c_2(\alpha)\phi_2(t),
\end{split}
\end{equation}
where reduction of order allows us to write $\phi_2(t)$ in terms of $\phi_1(t)$ as
$$\phi_2(t)=\phi_1(t)\int_0^t\frac{ds}{\phi_1^2(s)}.$$
Since $\dot{\omega}=-\lambda\gamma_\alpha^{-(\lambda+1)}\dot{\gamma_\alpha}$ by (\ref{eq:lin}) and $\gamma_\alpha(\alpha,0)=1,$ then $\omega(\alpha, 0)=1$ and $\dot{\omega}(\alpha, 0)=-\lambda u'_0(\alpha)$, from which $c_1(\alpha)$ and $c_2(\alpha)$ are obtained. Combining these results reduces (\ref{eq:unonhomosol2}) to
\begin{equation}
\label{eq:compat}
\begin{split}
\omega(\alpha,t)=\phi_1(t)\left(1-\lambda \eta(t)u_0'(\alpha)\right),\,\,\,\,\,\,\,\,\,\,\,\,\eta(t)=\int_0^t\frac{ds}{\phi_1^2(s)}.
\end{split}
\end{equation}
Now, (\ref{eq:lin}) and (\ref{eq:compat}) imply
\begin{equation}
\label{eq:jaco}
\begin{split}
\gamma_{\alpha}(\alpha,t)=\left(\phi_1(t)\mathcal{J}(\alpha,t)\right)^{-\frac{1}{\lambda}}, 
\end{split}
\end{equation}
where 
\begin{equation}
\label{eq:J}
\begin{split}
\mathcal{J}(\alpha,t)=1-\lambda\eta(t)u_0^\prime(\alpha),\,\,\,\,\,\,\,\,\,\,\,\,\,\,\,\,\mathcal{J}(\alpha,0)=1,
\end{split}
\end{equation}
however, uniqueness of solution to (\ref{eq:cha}) and periodicity of $u$ require
\begin{equation}
\label{eq:percha}
\begin{split}
\gamma(\alpha+1,t)-\gamma(\alpha,t)=1
\end{split}
\end{equation}
for as long as $u$ is defined. Spatially integrating (\ref{eq:jaco}) therefore yields 
\begin{equation}
\label{eq:phi1}
\phi_1(t)=\left(\int_0^1{\frac{d\alpha}{\mathcal{J}(\alpha,t)^{\frac{1}{\lambda}}}}\right)^{\lambda},
\end{equation}
and so, if we set
\begin{equation}
\label{eq:def}
\begin{split}
\mathcal{K}_i(\alpha, t)=\frac{1}{\mathcal{J}(\alpha,t)^{i+{\frac{1}{\lambda}}}},\,\,\,\,\,\,\,\,\,\,\bar{\mathcal{K}}_i(t)=\int_0^1{\frac{d\alpha}{\mathcal{J}(\alpha,t)^{i+\frac{1}{\lambda}}}}
\end{split}
\end{equation}
for $i=0,1,2,...,$ we can write $\gamma_\alpha$ in the form
\begin{equation}
\label{eq:sum}
\gamma_\alpha={\mathcal K}_0/{\bar{\mathcal K}}_0.
\end{equation}
As a result of using (\ref{eq:jacid}) and (\ref{eq:sum}), we obtain 
\begin{equation}
\label{eq:finalsolu}
\begin{split}
u_x(\gamma(\alpha,t),t)=
\dot\gamma_\alpha(\alpha, t)/\gamma_\alpha(\alpha,t)=(\ln({\mathcal{K}_0/\bar{\mathcal K}_0}))^{{}^.}\,.
\end{split}
\end{equation}
In addition, differentiating (\ref{eq:compat})ii) gives
\begin{equation}
\label{eq:etaivp}
\begin{split}
\dot{\eta}(t)=\left(\int_0^1{\frac{d\alpha}{\mathcal{J}(\alpha,t)^{\frac{1}{\lambda}}}}\right)^{-2\lambda},\,\,\,\,\,\,\,\,\,\,\,\eta(0)=0,
\end{split}
\end{equation}   
from which it follows that the existence of an eventual finite blow-up time $t_*>0$ will depend, in part, upon convergence of the integral
\begin{equation}
\label{eq:assympt}
\begin{split}
t(\eta)=\int_0^{\eta}{\left(\int_0^1{\frac{d\alpha}{(1-\lambda\mu u_0^\prime(\alpha))^{\frac{1}{\lambda}}}}\right)^{2\lambda}\,d\mu}
\end{split}
\end{equation}
as $\eta\uparrow\eta_*$ for $\eta_*>0$ to be defined. In an effort to simplify the following  arguments, we point out that (\ref{eq:finalsolu}) can be  rewritten in a slightly more useful form. The result is
\begin{equation}
\label{eq:mainsolu}
\begin{split}
u_x(\gamma(\alpha,t),t)=\frac{1}{\lambda\eta(t){\bar{\mathcal K}_0(t)}^{2\lambda}}
\left(\frac{1}{\mathcal{J}(\alpha, t)}-\frac{\bar{\mathcal{K}}_1(t)}{\bar{\mathcal K}_0(t)}
\right).
\end{split}
\end{equation}
This is derived as follows. From (\ref{eq:def}) and (\ref{eq:finalsolu}),
\begin{equation}
\label{eq:finalsolu22}
u_x(\gamma(\alpha,t),t)=\frac{1}{\bar{\mathcal{K}}_0(t)^{^{2\lambda}}}\left(\frac{u_0'(\alpha)}{\mathcal{J}(\alpha,t)}-\frac{1}{\bar{\mathcal{K}}_0(t)}\int_0^1{\frac{u_0'(\alpha)\,d\alpha}{\mathcal{J}(\alpha,t)^{1+\frac{1}{\lambda}}}}\right).
\end{equation}
However  
\begin{equation}
\label{eq:alt1}
\frac{u_0^\prime(\alpha)}{\mathcal{J}(\alpha,t)}=\frac{1}{\lambda\eta(t)}\left(\frac{1}{\mathcal{J}(\alpha,t)}-1\right),
\end{equation}
by (\ref{eq:J}), and so
\begin{equation}
\label{eq:alt2}
\int_0^1{\frac{u_0^\prime(\alpha)\,d\alpha}{\mathcal{J}(\alpha,t)^{1+\frac{1}{\lambda}}}}=\frac{\bar{\mathcal{K}}_1(t)-\bar{\mathcal{K}}_0(t)}{\lambda\eta(t)}.
\end{equation}
Substituting (\ref{eq:alt1}) and (\ref{eq:alt2}) into (\ref{eq:finalsolu22}) yields (\ref{eq:mainsolu}). Finally, assuming sufficient smoothness, we may use (\ref{eq:sum}) and (\ref{eq:mainsolu}) to obtain (\cite{Saxton1},\,\cite{Wunsch1})
\begin{equation}
\label{eq:preserv1}
\begin{split}
u_{xx}(\gamma(\alpha,t),t)=u_0^{\prime\prime}(\alpha)(\gamma_{\alpha}(\alpha,t))^{2\lambda-1}.
\end{split}
\end{equation}
Equation  (\ref{eq:preserv1}) implies that as long as a solution exists it will maintain its initial concavity profile. Also, since the exponent above changes sign through $\lambda=1/2,$ blow-up implies, relative to the value of $\lambda,$  either vanishing or divergence of the jacobian. More explicitly, (\ref{eq:sum}) and (\ref{eq:preserv1}) yield
\begin{equation}
\label{eq:seconder}
\begin{split}
u_{xx}(\gamma(\alpha,t),t)=\frac{u_0^{\prime\prime}(\alpha)}{\mathcal{J}(\alpha,t)^{2-\frac{1}{\lambda}}}\left(\int_0^1{\frac{d\alpha}{\mathcal{J}(\alpha,t)^{\frac{1}{\lambda}}}}\right)^{1-2\lambda}.
\end{split}
\end{equation}

\section{Global Estimates and Blow-up}
\label{sec:blow}
In \S \ref{subsubsec:alpha>1/2}-\ref{subsubsec:easyalpha} we establish Theorem \ref{thm:maintheorem1} and Corollary \ref{coro:coro1}, while Theorem \ref{thm:lpintro} is proved in \S \ref{subsubsec:lpregsmooth}. Theorems \ref{thm:pc0} and \ref{thm:pq} are proved in \S \ref{subsubsec:pc} and \S \ref{subsubsec:pl}, respectively.

For $M_0>0>m_0$ as in (\ref{eq:max}) and (\ref{eq:min22}), set
\begin{equation}
\label{eq:defeta*}
\begin{split}
\eta_*=
\begin{cases}
\frac{1}{\lambda M_0},\,\,\,\,\,\,\,&\lambda>0,
\\
\frac{1}{\lambda m_0},\,\,\,\,\,\,\,&\lambda<0.
\end{cases}
\end{split}
\end{equation}
Then, as $\eta\uparrow \eta_*,$ the space-dependent term in (\ref{eq:mainsolu}) will diverge for certain choices of $\alpha$ and not at all for others. Specifically, for $\lambda>0,$ $\mathcal{J}(\alpha,t)^{-1}$ blows up  earliest as $\eta\uparrow \eta_*$ at $\alpha=\overline\alpha_i,$ since
\begin{equation*}
\begin{split}
\mathcal{J}(\overline\alpha_i,t)^{-1}=(1-\lambda\eta(t)M_0)^{-1}\to+\infty\,\,\,\,\,\,\,\text{as}\,\,\,\,\,\,\,\eta\uparrow\eta_*=\frac{1}{\lambda M_0}.
\end{split}
\end{equation*}
Similarly for $\lambda<0,\, \mathcal{J}(\alpha,t)^{-1}$ diverges first at $\alpha=\underline\alpha_j$ and 
\begin{equation*}
\begin{split}
\mathcal{J}(\underline\alpha_j,t)^{-1}=(1-\lambda\eta(t)m_0)^{-1}\to+\infty\,\,\,\,\,\,\,\text{as}\,\,\,\,\,\,\,\eta\uparrow\eta_*=\frac{1}{\lambda m_0}.
\end{split}
\end{equation*}
However, blow-up of (\ref{eq:mainsolu}) does not necessarily follow from this; we will need to estimate the behaviour of the time-dependent integrals 
\begin{equation*}
\begin{split}
\bar{\mathcal{K}}_0(t)=
\int_0^1{\frac{d\alpha}{\mathcal{J}(\alpha,t)^{\frac{1}{\lambda}}}},\,\,\,\,\,\,\,\,\,\,\,\,\,\,\,\,\,\,\,\,\,\,\,\,\,\,\bar{\mathcal{K}}_1(t)=\int_0^1{\frac{d\alpha}{\mathcal{J}(\alpha,t)^{1+\frac{1}{\lambda}}}}
\end{split}
\end{equation*}
as $\eta\uparrow\eta_*.$ To this end, in some of the proofs we find convenient the use of the Gauss hypergeometric series (\cite{Barnes1}, \cite{Magnus1}, \cite{Gasper1})
\begin{equation}
\label{eq:2f1}
\begin{split}
{}_2F_1\left[a,b;c;z\right]\equiv\sum_{k=0}^{\infty}\frac{\left(a\right)_k(b)_k}{\left(c\right)_k\,k!}z^k,\,\,\,\,\,\,\,\,\,\,\,\,\,\,\,\lvert z\rvert< 1
\end{split}
\end{equation} 
for $c\notin\mathbb{Z}^-\cup\{0\}$ and $(x)_k,\, k\in\mathbb{N}\cup\{0\}$, the Pochhammer symbol $(x)_0=1$, $(x)_k=x(x+1)...(x+k-1).$ Also, we will make use of the following results (\cite{Magnus1}, \cite{Gasper1}):
\begin{proposition}
\label{lem:analcont}
Suppose $\lvert\text{arg}\left(-z\right)\rvert<\pi$ and $a,b,c,a-b\notin\mathbb{Z}.$ Then, the analytic continuation for $\lvert z\rvert>1$ of the series (\ref{eq:2f1}) is given by 
\begin{equation}
\label{eq:analform}
\begin{split}
{}_2F_1[a,b;c;z]=&\frac{\Gamma(c)\Gamma(a-b)(-z)^{-b}{}_2F_1[b,1+b-c;1+b-a;z^{-1}]}{\Gamma(a)\Gamma(c-b)}
\\
&+\frac{\Gamma(c)\Gamma(b-a)(-z)^{-a}{}_2F_1[a,1+a-c;1+a-b;z^{-1}]}{\Gamma(b)\Gamma(c-a)}
\end{split}
\end{equation} 
where $\Gamma(\cdot)$ denotes the standard gamma function.
\end{proposition}
\begin{lemma}
\label{prop:prop}
Suppose $b\in(-\infty,2)\backslash\{1/2\},\,\,0\leq\left|\beta-\beta_0\right|\leq1$ and $\epsilon\geq C_0$ for some $C_0>0.$ Then 
\begin{equation}
\label{eq:derseries}
\begin{split}
\frac{1}{\epsilon^b}\,\frac{d}{d\beta}\left((\beta-\beta_0){}_2F_1\left[\frac{1}{2},b;\frac{3}{2};-\frac{C_0(\beta-\beta_0)^2}{\epsilon}\right]\right)=(\epsilon+C_0(\beta-\beta_0)^2)^{-b}.
\end{split}
\end{equation} 
\end{lemma}
\begin{proof}
See appendix B.
\qquad\end{proof}

\subsection{A Class of Smooth Initial Data}
\label{subsec:mean}
In this section, we study finite-time blow-up of solutions to (\ref{eq:nonhomo})-(\ref{eq:pbc}) which arise from a class of mean-zero, smooth data. In \S \ref{subsubsec:alpha>1/2}, we consider parameter values $\lambda\in[0,+\infty)$ whereas the case $\lambda\in(-\infty,0)$ is studied in \S \ref{subsubsec:twocases} and \S \ref{subsubsec:easyalpha}. Finally, $L^p$ regularity of solutions is examined in \S \ref{subsubsec:lpregsmooth} for $p\in[1,+\infty)$. 

\subsubsection{\textbf{Global estimates for $\lambda\in[0,1]$ and blow-up for $\lambda\in(1,+\infty)$}}\hfill
\label{subsubsec:alpha>1/2}

In Theorem \ref{thm:theoremtwogen} below, we prove finite-time blow-up of $u_x$ in the $L^{\infty}$ norm for $\lambda\in(1,+\infty).$ In fact, we will find that the blow-up is two-sided and occurs everywhere in the domain, an event we will refer to as ``two-sided, everywhere blow-up.'' In contrast, for parameters $\lambda\in[0,1],$ we show that solutions persist globally in time. More particularly, these vanish as $t\to+\infty$ for $\lambda\in(0,1)$ but converge to a non-trivial steady-state if $\lambda=1.$ Finally, the behaviour of the jacobian (\ref{eq:sum}) is also studied. We refer to appendix A for the case $\lambda=0.$ 

\begin{theorem}
\label{thm:theoremtwogen}
Consider the initial boundary value problem (\ref{eq:nonhomo})-(\ref{eq:pbc}). There exist smooth, mean-zero initial data such that: 
\begin{enumerate}
	\item\label{it:case} For $\lambda\in(0,1],$ solutions persist globally in time. In particular, these vanish as $t\uparrow t_*=+\infty$ for $\lambda\in(0,1)$ but converge to a non-trivial steady-state if $\lambda=1.$
	\item\label{item:blowtwo2} For $\lambda\in(1,+\infty),$ there exists a finite $t_*>0$ such that both the maximum $M(t)$ and the minimum $m(t)$ diverge to $+\infty$ and respectively to $-\infty$ as $t\uparrow t_*.$ Moreover, $\lim_{t\uparrow t_*}u_x(\gamma(\alpha,t),t)=-\infty$ for $\alpha\notin\{\overline\alpha_i, \underline\alpha_j\}$ (two-sided, everywhere blow-up).
\end{enumerate}
Finally, for $t_*$ as above, the jacobian (\ref{eq:sum}) satisfies 
\begin{equation}
\label{eq:sumjacblowposlamb}
\lim_{t\uparrow t_*}\gamma_\alpha(\alpha,t)=
\begin{cases}
+\infty,\,\,\,\,\,\,\,\,\,\,\,\,\,&\alpha=\overline\alpha_i,\,\,\,\,\,\,\,\,\,\,\,\,\,\,\,\,\,\,\,\,\,\,\lambda\in(0,+\infty),
\\
0,\,\,\,\,\,\,\,\,\,\,\,\,\,&\alpha\neq\overline\alpha_i,\,\,\,\,\,\,\,\,\,\,\,\,\,\,\,\,\,\,\,\,\,\,\lambda\in(0,2],
\\
C,\,\,\,\,\,\,\,\,\,\,\,\,\,&\alpha\neq\overline\alpha_i,\,\,\,\,\,\,\,\,\,\,\,\,\,\,\,\,\,\,\,\,\,\lambda\in(2,+\infty)
\end{cases}
\end{equation}
where the positive constants $C$ depend on the choice of $\lambda$ and $\alpha\neq\overline\alpha_i.$
\end{theorem}
\begin{proof}
For simplicity, assume $M_0>0$ is attained at a single location\footnote[3]{The case of a finite number of $\overline\alpha_i\in[0,1]$ follow similarly.} $\overline\alpha\in(0,1).$ We consider the case where, near $\overline\alpha,\, u_0^\prime(\alpha)$ has non-vanishing second order derivative, so that, locally $u_0^\prime(\alpha)\sim M_0+C_1(\alpha-\overline\alpha)^2$ for $0\leq\lvert\alpha-\overline\alpha\rvert\leq s,$ $0<s\leq1$ and $C_1=u_0^{\prime\prime\prime}(\overline\alpha)/2<0.$ Then, for $\epsilon>0$ 
\begin{equation}
\label{eq:exp}
\begin{split}
\epsilon-u_0^\prime(\alpha)+M_0\sim\epsilon-C_1(\alpha-\overline\alpha)^2.
\end{split}
\end{equation}

\textsl{\textit{\textbf{Global existence for $\lambda\in(0,1].$}}}
\vspace{0.02in}

By (\ref{eq:exp}) above and the change of variables $\alpha=\sqrt{\frac{\epsilon}{\left|C_1\right|}}\tan\theta+\overline\alpha$, we have that 
\begin{equation}
\label{eq:comple1}
\begin{split}
\int_{\overline\alpha-s}^{\overline\alpha+s}{\frac{d\alpha}{(\epsilon-C_1(\alpha-\overline\alpha)^2)^{\frac{1}{\lambda}}}}\sim\frac{\epsilon^{\frac{1}{2}-\frac{1}{\lambda}}}{\sqrt{-C_1}}\int_{-\frac{\pi}{2}}^{\frac{\pi}{2}}{(\text{cos}(\theta))^{2\left(\frac{1}{\lambda}-1\right)}d\theta}
\end{split}
\end{equation}
for $\epsilon>0$ small and $\lambda\in(0,1]$. But from properties of the Gamma function (see for instance \cite{Gamelin1}), the identity 
\begin{equation}
\label{eq:gamma1}
\begin{split}
\int_0^1{t^{p-1}(1-t)^{q-1}dt}=\frac{\Gamma(p)\,\Gamma(q)}{\Gamma(p+q)}
\end{split}
\end{equation}
holds for all $p,q>0$. Therefore, setting $p=\frac{1}{2}$, $q=\frac{1}{\lambda}-\frac{1}{2}$ and $t=\sin^2\theta$ into (\ref{eq:gamma1}) gives
$$\int_{-\frac{\pi}{2}}^{\frac{\pi}{2}}{(\text{cos}(\theta))^{2\left(\frac{1}{\lambda}-1\right)}d\theta}=\frac{\sqrt{\pi}\,\Gamma\left(\frac{1}{\lambda}-\frac{1}{2}\right)}{\Gamma\left(\frac{1}{\lambda}\right)},$$
which we use, along with (\ref{eq:exp}) and (\ref{eq:comple1}), to obtain
\begin{equation}
\label{eq:comple2}
\begin{split}
\int_{0}^{1}{\frac{d\alpha}{(\epsilon-u_0^\prime(\alpha)+M_0)^{\frac{1}{\lambda}}}}\sim \frac{\Gamma\left(\frac{1}{\lambda}-\frac{1}{2}\right)}{\Gamma\left(\frac{1}{\lambda}\right)}\sqrt{-\frac{\pi}{C_1}}\,\epsilon^{\frac{1}{2}-\frac{1}{\lambda}}.
\end{split}
\end{equation}
Consequently, setting $\epsilon=\frac{1}{\lambda\eta}-M_0$ into (\ref{eq:comple2}) yields
\begin{equation}
\label{eq:comple3}
\begin{split}
\bar{\mathcal K}_0(t)\sim C_3\mathcal{J}(\overline\alpha,t)^{\frac{1}{2}-\frac{1}{\lambda}}
\end{split}
\end{equation}
for $\eta_*-\eta>0$ small, $\mathcal{J}(\overline\alpha,t)=1-\lambda\eta(t)M_0$, $\eta_*=\frac{1}{\lambda M_0}$ and positive constants $C_3$ given by
\begin{equation}
\label{eq:constgen2}
C_3=\frac{\Gamma\left(\frac{1}{\lambda}-\frac{1}{2}\right)}{\Gamma\left(\frac{1}{\lambda}\right)}\sqrt{-\frac{\pi M_0}{C_1}}.
\end{equation} 
Similarly, 
\begin{equation}
\label{eq:comple3.1}
\begin{split}
\int_{\overline\alpha-s}^{\overline\alpha+s}{\frac{d\alpha}{(\epsilon-C_1(\alpha-\overline\alpha)^2)^{1+\frac{1}{\lambda}}}}\sim\frac{\Gamma\left(\frac{1}{2}+\frac{1}{\lambda}\right)}{\Gamma\left(1+\frac{1}{\lambda}\right)}\sqrt{-\frac{\pi}{C_1}}\,\epsilon^{-\left(\frac{1}{2}+\frac{1}{\lambda}\right)}
\end{split}
\end{equation}
so that
\begin{equation}
\label{eq:alphaposgen7}
\begin{split}
\bar{\mathcal{K}}_1(t)\sim \frac{C_4}{\mathcal{J}(\overline\alpha,t)^{\frac{1}{2}+\frac{1}{\lambda}}}
\end{split}
\end{equation}
for $\lambda\in(0,1]$ and positive constants $C_4$ determined by
\begin{equation}
\label{eq:alphaposgen8}
\begin{split}
C_4=\frac{\Gamma\left(\frac{1}{2}+\frac{1}{\lambda}\right)}{\Gamma\left(1+\frac{1}{\lambda}\right)}\sqrt{-\frac{\pi M_0}{C_1}}.
\end{split}
\end{equation}
Using (\ref{eq:comple3}) and (\ref{eq:alphaposgen7}) with (\ref{eq:mainsolu}) implies
\begin{equation}
\label{eq:finalgen1}
\begin{split}
u_x(\gamma(\alpha,t),t)\sim\frac{C}{\mathcal{J}(\overline\alpha,t)^{\lambda-1}}\left(\frac{\mathcal{J}(\overline\alpha,t)}{\mathcal{J}(\alpha,t)}-\frac{C_4}{C_3}\right)
\end{split}
\end{equation}
for $\eta_*-\eta>0$ small. But $\Gamma(y+1)=y\,\Gamma(y)$, $y\in\mathbb{R}^+$ (see e.g. \cite{Gamelin1}), so that 
\begin{equation}
\label{eq:simplifygamma}
\begin{split}
\frac{C_4}{C_3}=\frac{\Gamma\left(\frac{1}{\lambda}\right)\,\Gamma\left(\frac{1}{\lambda}-\frac{1}{2}+1\right)}{\Gamma\left(\frac{1}{\lambda}+1\right)\,\Gamma\left(\frac{1}{\lambda}-\frac{1}{2}\right)}=1-\frac{\lambda}{2}\in[1/2,1)
\end{split}
\end{equation}
for $\lambda\in(0,1]$. Then, by (\ref{eq:finalgen1}), (\ref{eq:maxp})i) and the definition of $M_0$
\begin{equation}
\label{eq:globalest2}
\begin{split}
M(t)\to0^+,\,\,\,\,\,\,\,\,\,\,\,&\alpha=\overline\alpha,
\\
u_x(\gamma(\alpha,t),t)\to0^-,\,\,\,\,\,\,\,\,\,\,\,\,&\alpha\neq\overline\alpha
\end{split}
\end{equation}
as $\eta\uparrow\eta_*$ for all $\lambda\in(0,1)$. For the threshold parameter $\lambda_*=1,$ we keep track of the positive constant $C$ prior to (\ref{eq:finalgen1}) and find that, for $\alpha=\overline\alpha,$
\begin{equation}
\label{eq:oneside}
\begin{split}
M(t)\to-\frac{u_0^{\prime\prime\prime}(\overline\alpha)}{(2\pi)^2}>0
\end{split}
\end{equation}
as $\eta\uparrow\frac{1}{M_0},$ whereas 
\begin{equation}
\label{eq:otherside}
\begin{split}
u_x(\gamma(\alpha,t),t)\to\frac{u_0^{\prime\prime\prime}(\overline\alpha)}{(2\pi)^2}<0
\end{split}
\end{equation}
for $\alpha\neq\overline\alpha.$ Finally, from (\ref{eq:etaivp}) 
\begin{equation}
\label{eq:ass1}
\begin{split}
dt=\bar{\mathcal{K}}_0(t)^{2\lambda}d\eta,
\end{split}
\end{equation}
then (\ref{eq:comple3}) implies
\begin{equation}
\label{eq:ass2}
\begin{split}
t_*-t\sim C\int_\eta^{\eta_*}{(1-\lambda\mu M_0)^{\lambda-2}d\mu}.
\end{split}
\end{equation}
As a result, $t_*=+\infty$ for all $\lambda\in(0,1].$ See \S \ref{sec:examples} for examples.
\vspace{0.02in}

\textsl{\textit{\textbf{Two-sided, everywhere blow-up for $\lambda\in(1,+\infty).$}}}
\vspace{0.02in}

For $\lambda\in(1,+\infty)\backslash\{2\}$, set $b=\frac{1}{\lambda}$ in Lemma \ref{prop:prop} to obtain
\begin{equation}
\label{eq:gen2f12}
\begin{split}
\int_{\overline\alpha-s}^{\overline\alpha+s}{\frac{d\alpha}{(\epsilon-C_1(\alpha-\overline\alpha)^2)^{\frac{1}{\lambda}}}}=2s\epsilon^{-\frac{1}{\lambda}}\,{}_2F_1\left[\frac{1}{2},\frac{1}{\lambda};\frac{3}{2};\frac{s^2C_1}{\epsilon}\right]
\end{split} 
\end{equation}
where the above series is defined by (\ref{eq:2f1}) as long as $\epsilon\geq-C_1\geq-s^2C_1>0$, namely $-1\leq\frac{s^2C_1}{\epsilon}<0$. However, we are ultimately interested in the behaviour of (\ref{eq:gen2f12}) for $\epsilon>0$ arbitrarily small, so that, eventually $\frac{s^2C_1}{\epsilon}<-1$. To achieve this transition of the series argument across $-1$ in a well-defined, continuous fashion, we use proposition \ref{lem:analcont} which provides us with the analytic continuation of the series in (\ref{eq:gen2f12}) from argument values inside the unit circle, in particular for the interval 
%$[-1,1]$, i.e.
 $-1\leq\frac{s^2C_1}{\epsilon}<0$, to those found outside and thus for $\frac{s^2C_1}{\epsilon}<-1$. Consequently, for $\epsilon$ small enough, so that $-s^2C_1>\epsilon>0$, proposition \ref{lem:analcont} implies
\begin{equation}
\label{eq:gen2f122}
\begin{split}
2s\epsilon^{-\frac{1}{\lambda}}\,{}_2F_1\left[\frac{1}{2},\frac{1}{\lambda};\frac{3}{2};\frac{s^2C_1}{\epsilon}\right]= C\,\Gamma\left(\frac{1}{\lambda}-\frac{1}{2}\right)\epsilon^{\frac{1}{2}-\frac{1}{\lambda}}+\frac{C}{\lambda-2}+\psi(\epsilon)
\end{split}
\end{equation}
%Setting $b=\frac{1}{\lambda}$ into (\ref{eq:derseries}) and using lemma \ref{prop:prop} gives 
%\begin{equation}
%\label{eq:gen2f12}
%\begin{split}
%\int_{\overline\alpha-s}^{\overline\alpha+s}{\frac{d\alpha}{(\epsilon-C_1(\alpha-\overline\alpha)^2)^{\frac{1}{\lambda}}}}=2s\epsilon^{-\frac{1}{\lambda}}\,{}_2F_1\left[\frac{1}{2},\frac{1}{\lambda};\frac{3}{2};\frac{s^2C_1}{\epsilon}\right]
%\end{split}
%\end{equation}
%for $\epsilon\geq-C_1>0$ and $\lambda\in(1,+\infty)\backslash\{2\}.$ But, taking $\epsilon$ small enough, eventually $\frac{s^2C_1}{\epsilon}<-1.$ Then proposition \ref{lem:analcont} implies
%\begin{equation}
%\label{eq:gen2f122}
%\begin{split}
%2s\epsilon^{-\frac{1}{\lambda}}\,{}_2F_1\left[\frac{1}{2},\frac{1}{\lambda};\frac{3}{2};\frac{s^2C_1}{\epsilon}\right]= C\,\Gamma\left(\frac{1}{\lambda}-\frac{1}{2}\right)\epsilon^{\frac{1}{2}-\frac{1}{\lambda}}+\frac{C}{\lambda-2}+\psi(\epsilon)
%\end{split}
%\end{equation}
for $\psi(\epsilon)=\textsl{o}(1)$ as $\epsilon\to0$ and positive constant $C$ which may depend on $\lambda$ and can be obtained explicitly from (\ref{eq:analform}). Then, substituting $\epsilon=\frac{1}{\lambda\eta}-M_0$ into (\ref{eq:gen2f122}) and using (\ref{eq:exp}) along with (\ref{eq:gen2f12}), yields
\begin{equation}
\label{eq:alphaposgen6}
\begin{split}
\bar{\mathcal{K}}_0(t)\sim
\begin{cases}
C_3\mathcal{J}(\overline\alpha,t)^{\frac{1}{2}-\frac{1}{\lambda}},\,\,\,&\lambda\in(1,2),
\\
C,\,\,&\lambda\in(2,+\infty)
\end{cases}
\end{split}
\end{equation}
for $\eta_*-\eta>0$ small and positive constants $C_3$ given by (\ref{eq:constgen2}) for $\lambda\in(1,2)$. Similarly, by following an identical argument, with $b=1+\frac{1}{\lambda}$ instead, we find that estimate (\ref{eq:alphaposgen7}), derived initially for $\lambda\in(0,1]$, holds for $\lambda\in(1,+\infty)$ as well. First suppose $\lambda\in(1,2)$, then (\ref{eq:mainsolu}), (\ref{eq:alphaposgen7}) and (\ref{eq:alphaposgen6})i) imply estimate (\ref{eq:finalgen1}). However, by (\ref{eq:simplifygamma}) we now have 
$$\frac{C_4}{C_3}=1-\frac{\lambda}{2}\in(0,1/2)$$ for $\lambda\in(1,2)$.
%\begin{equation}
%\label{eq:finalgen1}
%\begin{split}
%u_x(\gamma(\alpha,t),t)\sim \frac{C}{\mathcal{J}(\overline\alpha,t)^{\lambda-1}}\left(\frac{\mathcal{J}(\overline\alpha,t)}{\mathcal{J}(\alpha,t)}-C\right).
%\end{split}
%\end{equation}
As a result, setting $\alpha=\overline\alpha$ in (\ref{eq:finalgen1}), we obtain
\begin{equation}
\label{eq:finalgen2}
\begin{split}
M(t)\sim\frac{C}{\mathcal{J}(\overline\alpha,t)^{\lambda-1}}\to+\infty
\end{split}
\end{equation}
as $\eta\uparrow\eta_*$. On the other hand, if $\alpha\neq\overline\alpha,$ the definition of $M_0$ gives 
\begin{equation}
\label{eq:finalgen3}
\begin{split}
u_x(\gamma(\alpha,t),t)\sim-\frac{C}{\mathcal{J}(\overline\alpha,t)^{\lambda-1}}\to-\infty.
\end{split}
\end{equation}
The existence of a finite $t_*>0$ follows from (\ref{eq:ass1}) and (\ref{eq:alphaposgen6})i), which imply
\begin{equation}
\label{eq:ass3}
\begin{split}
t_*-t\sim C(\eta_*-\eta)^{\lambda-1}.
\end{split}
\end{equation}
For $\lambda\in(2,+\infty),$ we use (\ref{eq:mainsolu}), (\ref{eq:alphaposgen7}) and (\ref{eq:alphaposgen6})ii) to get
\begin{equation}
\label{eq:finalgenother}
\begin{split}
u_x(\gamma(\alpha,t),t)\sim \frac{C}{\mathcal{J}(\overline\alpha,t)}\left(\frac{\mathcal{J}(\overline\alpha,t)}{\mathcal{J}(\alpha,t)}-C\mathcal{J}(\overline\alpha,t)^{\frac{1}{2}-\frac{1}{\lambda}}\right).
\end{split}
\end{equation}
Then, setting $\alpha=\overline\alpha$ in (\ref{eq:finalgenother}), we obtain
\begin{equation}
\label{eq:finalgen4}
\begin{split}
M(t)\sim\frac{C}{\mathcal{J}(\overline\alpha,t)}\to+\infty
\end{split}
\end{equation}
as $\eta\uparrow\eta_*.$ Similarly, for $\alpha\neq\overline\alpha,$ 
\begin{equation}
\label{eq:finalgen5}
\begin{split}
u_x(\gamma(\alpha,t),t)\sim-\frac{C}{\mathcal{J}(\overline\alpha,t)^{\frac{1}{2}+\frac{1}{\lambda}}}\to-\infty.
\end{split}
\end{equation}
A finite blow-up time $t_*>0$ follows from (\ref{eq:ass1}) and (\ref{eq:alphaposgen6})ii), which yield $$t_*-t\sim C(\eta_*-\eta).$$ For the case $\lambda=2$ and $\eta_*-\eta=\frac{1}{2M_0}-\eta>0$ small, we have 
\begin{equation}
\label{eq:genalpha2}
\begin{split}
\bar{\mathcal{K}}_0(t)\sim-C\ln\left(\mathcal{J}(\overline\alpha,t)\right),\,\,\,\,\,\,\,\,\,\,\,\,\,\,\,\,\,\,\,\,\,\bar{\mathcal{K}}_1(t)\sim\frac{C}{\mathcal{J}(\overline\alpha,t)}.
\end{split}
\end{equation}
Two-sided blow-up for $\lambda=2$ then follows from (\ref{eq:mainsolu}), (\ref{eq:ass1}) and (\ref{eq:genalpha2}). Finally, the behaviour of the jacobian in (\ref{eq:sumjacblowposlamb}) is deduced from (\ref{eq:sum}) and the estimates (\ref{eq:comple3}), (\ref{eq:alphaposgen6}) and (\ref{eq:genalpha2}). See \S \ref{subsec:smooth} for examples.\qquad\end{proof}
\begin{remark}
\label{rem:nonsmooth}
Several methods were used in \cite{Childress} to show that there are blow-up solutions for $\lambda=1$ under Dirichlet boundary conditions. We remark that these do not conflict with our global  result in part \ref{it:case} of Theorem \ref{thm:theoremtwogen} as long as the data is smooth and, under certain circumstances, its local behaviour near the endpoints $\alpha=\{0,1\}$ allows for a smooth, periodic extension of $u_0'$ to all $\alpha\in\mathbb{R}.$ Further details on this will be given in future work. See also \S \ref{subsubsec:pl} where a particular choice of $u_0''(\alpha)\in PC_{\mathbb{R}}(0,1)$ leads to finite-time blow-up for all $\lambda\in(1/2,+\infty).$
\end{remark}

\subsubsection{\textbf{Blow-up for $\lambda\in(-\infty,-1)$}}\hfill
\label{subsubsec:twocases}

Theorem \ref{thm:theoremfour} below shows the existence of mean-zero, smooth data for which solutions undergo a two-sided, everywhere blow-up in finite-time for $\lambda\in(-\infty,-2]$, whereas, if $\lambda\in(-2,-1),$ only the minimum diverges.
\begin{theorem}
\label{thm:theoremfour}
Consider the initial boundary value problem (\ref{eq:nonhomo})-(\ref{eq:pbc}). There exist smooth, mean-zero initial data such that:
\begin{enumerate}
\item\label{item:twothm4} For $\lambda\in(-\infty,-2],$ there is a finite $t_*>0$ such that both the maximum $M(t)$ and the minimum $m(t)$ diverge to $+\infty$ and respectively to $-\infty$ as $t\uparrow t_*.$ In addition, $\lim_{t\uparrow t_*}u_x(\gamma(\alpha,t),t)=+\infty$ for $\alpha\notin\{\overline\alpha_i, \underline\alpha_j\}$ (two-sided, everywhere blow-up). 
\item\label{item:onethm4} For $\lambda\in(-2,-1),$ there exists a finite $t_*>0$ such that only the minimum diverges, $m(t)\to-\infty,$ as $t\uparrow t_*$ (one-sided, discrete blow-up). 
\end{enumerate}
Finally, for $\lambda\in(-\infty,-1)$ and $t_*$ as above, the jacobian (\ref{eq:sum}) satisfies
\begin{equation}
\label{eq:sumjacblowalpha-infty-1}
\lim_{t\uparrow t_*}\gamma_\alpha(\alpha,t)=
\begin{cases}
0,\,\,\,\,\,\,\,\,\,\,\,\,\,&\alpha=\underline\alpha_j,
\\
C,\,\,\,\,\,\,\,\,\,\,\,\,\,&\alpha\neq\underline\alpha_j
\end{cases}
\end{equation}
where the positive constants $C$ depend on the choice of $\lambda$ and $\alpha\neq\underline\alpha_j.$
\end{theorem}
\begin{proof}
For $\lambda\in(-\infty,-1),$ smoothness of $u_0'$ implies that $\bar{\mathcal{K}}_0(t)=\int_0^1{\mathcal{J}(\alpha,t)^{\frac{1}{\left|\lambda\right|}}d\alpha}$ remains finite (and positive) for all $\eta\in[0,\eta_*)$,\, $\eta_*=\frac{1}{\lambda m_0}.$ Also, $\bar{\mathcal{K}}_0(t)$ has a finite, positive limit as $\eta\uparrow \eta_*.$ Indeed, suppose there is an earliest $t_1>0$ such that $\eta_1=\eta(t_1)>0$ and 
\begin{equation}
\label{eq:nun}
\begin{split}
\bar{\mathcal{K}}_0(t_1)=\int_0^1{(1-\lambda\eta_1u_0^\prime(\alpha))^{\frac{1}{\left|\lambda\right|}}d\alpha}=0.
\end{split}
\end{equation}
Since $\int_0^1{\left(1-u_0^\prime/m_0\right)^{\frac{1}{\left|\lambda\right|}}d\alpha}>0$, then $\eta_1\neq\eta_*.$ Also, by periodicity of $u_0$, there are $[0,1]\ni\alpha_1\neq\underline\alpha_j$ where $(1-\lambda\eta_1u_0^\prime(\alpha_1))^{\frac{1}{\left|\lambda\right|}}=1$, and so, (\ref{eq:nun}) implies the existence of at least one $\alpha^\prime\neq\underline\alpha_j$ where $u_0^\prime(\alpha^\prime)=\frac{1}{\lambda\eta_1}$. But $u_0^\prime\geq m_0$ and $\eta_*=\frac{1}{\lambda m_0},$ then
\begin{equation}
\label{eq:back3}
\begin{split}
\eta_*<\eta_1.
\end{split}
\end{equation}
In addition, (\ref{eq:back3}) and $m_0\leq u_0^\prime\leq M_0$ yield 
\begin{equation}
\label{eq:newest}
\begin{split}
0<\int_0^1{\frac{d\alpha}{(1-u_0^\prime/m_0)^{\frac{1}{\lambda}}}}\leq\bar{\mathcal{K}}_0(t)\leq1,\,\,\,\,\,\,\,\,\,\,0\leq\eta\leq\eta_*.
\end{split}
\end{equation} 
Next, for $\lambda\in(-\infty,-1),$ we estimate $\bar{\mathcal{K}}_1(t)=\int_0^1{\mathcal{J}(\alpha,t)^{\frac{1}{\left|\lambda\right|}-1}d\alpha}$ as $\eta\uparrow\eta_*$ by following an argument similar to that of Theorem \ref{thm:theoremtwogen}. For simplicity, suppose $m_0$ occurs at a single $\underline\alpha\in(0,1).$ We consider the case where, near $\underline\alpha,\, u_0^\prime(\alpha)$ has non-vanishing second order derivative, so that, locally $u_0^\prime(\alpha)\sim m_0+C_2(\alpha-\underline\alpha)^2$ for $0\leq\lvert \alpha-\underline\alpha\rvert\leq r,$ $0<r\leq1$ and $C_2=u_0'''(\underline\alpha)/2>0.$ Then, for $\epsilon>0$
\begin{equation}
\label{eq:othere}
\begin{split}
\epsilon+u_0^\prime(\alpha)-m_0\sim\epsilon+C_2(\alpha-\underline\alpha)^2.
\end{split}
\end{equation} 
Given $\lambda\in(-\infty,-1),$ set $b=1+\frac{1}{\lambda}$ in Lemma (\ref{prop:prop}) to find
\begin{equation}
\label{eq:toit}
\begin{split}
\int_{\underline\alpha-r}^{\underline\alpha+r}{\frac{d\alpha}{(\epsilon+C_2(\alpha-\underline\alpha)^2)^{1+\frac{1}{\lambda}}}}=\frac{2r}{\epsilon^{1+\frac{1}{\lambda}}}\,{}_2F_1\left[\frac{1}{2},1+\frac{1}{\lambda};\frac{3}{2};-\frac{r^2C_2}{\epsilon}\right]
\end{split}
\end{equation}
for $\epsilon\geq C_2\geq r^2C_2$, i.e. $-1\leq-\frac{r^2C_2}{\epsilon}<0$, and $\lambda\in(-\infty,-1)\backslash\{-2\}$\footnote[4]{The case $\lambda=-2$ is treated separately.}. Then, as we let $\epsilon>0$ become small enough, so that eventually $-\frac{r^2C_2}{\epsilon}<-1$, Proposition \ref{lem:analcont} implies
\begin{equation}
\label{eq:lastser}
\begin{split}
\frac{2r}{\epsilon^{1+\frac{1}{\lambda}}}\,{}_2F_1\left[\frac{1}{2},1+\frac{1}{\lambda};\frac{3}{2};-\frac{r^2C_2}{\epsilon}\right]=\frac{C}{\lambda+2}+\frac{C\,\Gamma\left(\frac{1}{2}+\frac{1}{\lambda}\right)}{\epsilon^{\frac{1}{2}+\frac{1}{\lambda}}}+\xi(\epsilon)
\end{split}
\end{equation}
for $\xi(\epsilon)=\textsl{o}(1)$ as $\epsilon\to0$ and positive constants $C$ which may depend on the choice of $\lambda$ and can be obtained explicitly from (\ref{eq:analform}). Using (\ref{eq:lastser}) on (\ref{eq:toit}), along with (\ref{eq:othere}) and the substitution $\epsilon=m_0-\frac{1}{\lambda\eta},$ yields
\begin{equation}
\label{eq:lastfin}
\bar{\mathcal{K}}_1(t)\sim
\begin{cases}
C,\,\,\,\,\,\,\,\,\,&\lambda\in(-2,-1),
\\
C_5\mathcal{J}(\underline\alpha,t)^{-\left(\frac{1}{2}+\frac{1}{\lambda}\right)},\,\,\,\,\,\,\,\,\,\,\,&\lambda\in(-\infty,-2)
\end{cases}
\end{equation}
for $\eta_*-\eta>0$ small, $\eta_*=\frac{1}{\lambda m_0},\,\, \mathcal{J}(\underline\alpha,t)=1-\lambda\eta(t)m_0$ and
\begin{equation}
\label{eq:lastconst}
\begin{split}
C_5=\frac{\Gamma\left(\frac{1}{2}+\frac{1}{\lambda}\right)}{\Gamma\left(1+\frac{1}{\lambda}\right)}\sqrt{-\frac{\pi m_0}{C_2}}>0,\,\,\,\,\,\,\,\,\,\,\,\,\,\lambda\in(-\infty,-2).
\end{split}
\end{equation}
Setting $\alpha=\underline\alpha$ in (\ref{eq:mainsolu}) and using (\ref{eq:maxp})ii), (\ref{eq:newest}) and (\ref{eq:lastfin}), implies
\begin{equation}
\label{eq:lastmin}
\begin{split}
m(t)\sim-\frac{C}{\mathcal{J}(\underline\alpha,t)}\to-\infty
\end{split}
\end{equation}
as $\eta\uparrow\eta_*$ for all $\lambda\in(-\infty,-1)\backslash\{-2\}.$ On the other hand, using (\ref{eq:mainsolu}), (\ref{eq:newest}), (\ref{eq:lastfin}) and the definition of $m_0,$ we see that, for $\alpha\neq\underline\alpha,$ 
\begin{equation}
\label{eq:lastother1}
\begin{cases}
\left|u_x(\gamma(\alpha,t),t)\right|<+\infty,\,\,\,\,\,\,\,&\lambda\in(-2,-1),
\\
u_x(\gamma(\alpha,t),t)\sim C\mathcal{J}(\underline\alpha,t)^{-\left(\frac{1}{2}+\frac{1}{\lambda}\right)}\to+\infty,\,\,\,\,\,\,\,\,\,\,&\lambda\in(-\infty,-2)
\end{cases}
\end{equation}
as $\eta\uparrow\eta_*.$ A one-sided, discrete blow-up for $\lambda\in(-2,-1)$ follows from (\ref{eq:lastmin}) and (\ref{eq:lastother1})i), whereas a two-sided, everywhere blow-up for $\lambda\in(-\infty,-2)$ results from (\ref{eq:lastmin}) and (\ref{eq:lastother1})ii). The existence of a finite $t_*>0$ follows from (\ref{eq:etaivp}) and (\ref{eq:newest}) as $\eta\uparrow \eta_*.$ Particularly, we have the lower bound
\begin{equation}
\label{eq:lowt}
\begin{split}
\eta_*\leq t_*<+\infty.
\end{split}
\end{equation} 
The case $\lambda=-2$ can be treated directly. We find
\begin{equation}
\label{eq:alpha-2}
\begin{split}
\bar{\mathcal{K}}_1(t)\sim-C\ln\left(\mathcal{J}(\underline\alpha,t)\right)
\end{split}
\end{equation}
for $\eta_*-\eta>0$ small. A two-sided blow-up then follows as above. Finally, (\ref{eq:sumjacblowalpha-infty-1}) is deduced from (\ref{eq:sum}) and (\ref{eq:newest}). See \S \ref{subsec:smooth} for examples.\qquad\end{proof}

\subsubsection{\textbf{One-sided, discrete blow-up for $\lambda\in[-1,0)$}}\hfill
\label{subsubsec:easyalpha}

Since (\ref{eq:jacid}) and $\gamma_{\alpha}(\alpha,0)=1$ imply the existence of a time interval $[0,t_*)$ where
\begin{equation}
\label{eq:alt}
\begin{split}
\gamma_\alpha(\alpha,t)=\exp\left(\int_0^t{u_x(\gamma(\alpha,s),s)\,ds}\right)>0
\end{split}
\end{equation}
for $\alpha\in[0,1]$ and $0<t_*\leq+\infty$, (\ref{eq:jacid2}) implies that for $\lambda\in[-1,0)$, $(\gamma_{\alpha}^{-\lambda}(\alpha,t))^{\ddot{}}\leq0.$ But $(\gamma_\alpha^{-\lambda})^{\cdot}|_{t=0}=-\lambda u_0^\prime,$ thus integrating twice in time gives $$\gamma_{\alpha}(\alpha,t)^{-\lambda}\leq 1-\lambda t\,u_0^\prime(\alpha).$$ Provided there is $\underline\alpha\in[0,1]$ such that $\inf_{\alpha\in[0,1]}{u_0^{\prime}(\alpha)}=u_0^{\prime}(\underline\alpha)<0,$ we define $t_*=(\lambda u_0^\prime(\underline\alpha))^{-1},$ then $\gamma_{\alpha}(t,\underline\alpha)\downarrow0$ as $t\uparrow t_*$ for any $\lambda\in[-1,0).$ This along with (\ref{eq:maxp})ii) and (\ref{eq:alt}) implies 
\begin{equation}
\label{eq:alt5}
\begin{split}
\lim_{t\uparrow t_*}{\int_0^t{u_x(\gamma(\underline\alpha,s),s)\,ds}}=\lim_{t\uparrow t_*}{\int_0^t{m(s)ds}}=-\infty.
\end{split}
\end{equation}
More precise blow-up properties are now studied via formula (\ref{eq:mainsolu}). Theorem \ref{thm:theoremthree} below will extend the one-sided, discrete blow-up found in Theorem \ref{thm:theoremfour} for parameters $\lambda\in(-2,-1)$ to all $\lambda\in(-2,0).$  
\begin{theorem}
\label{thm:theoremthree}
Consider the initial boundary value problem (\ref{eq:nonhomo})-(\ref{eq:pbc}) with arbitrary smooth, mean-zero initial data. For every $\lambda\in[-1,0),$ there exists a finite $t_*>0$ such that only the minimum diverges, $m(t)\to-\infty,$ as $t\uparrow t_*$ (one-sided, discrete blow-up). Also, the jacobian (\ref{eq:sum}) satisfies
\begin{equation}
\label{eq:sumjacblowalpha-10}
\lim_{t\uparrow t_*}\gamma_\alpha(\alpha,t)=
\begin{cases}
0,\,\,\,\,\,\,\,\,\,\,\,\,\,&\alpha=\underline\alpha_j,
\\
C,\,\,\,\,\,\,\,\,\,\,\,\,\,&\alpha\neq\underline\alpha_j
\end{cases}
\end{equation}
where the positive constants $C$ depend on the choice of $\lambda$ and $\alpha\neq\underline\alpha_j.$
\end{theorem}
\begin{proof}
Since $u_0'$ is smooth and $\lambda\in[-1,0),$ both integrals $\bar{\mathcal{K}}_i(t),$ $i=0,1$ remain finite (and positive) for all $\eta\in[0,\eta_*),\, \eta_*=\frac{1}{\lambda m_0}.$ Also, $\bar{\mathcal{K}}_0(t)$ does not vanish as $\eta\uparrow\eta_*.$ In fact
\begin{equation}
\label{eq:quick3}
\begin{split}
1\leq\bar{\mathcal{K}}_0(t)\leq\left(1-\frac{M_0}{m_0}\right)^{\frac{1}{\left|\lambda\right|}}
\end{split}
\end{equation}
for all $\eta\in[0,\eta_*].$ Indeed, notice that $\dot{\bar{\mathcal{K}}}_0(0)=0$ and
\begin{equation*}
\begin{split}
\ddot{\bar{\mathcal{K}}}_0(t)=\left((1+\lambda)\int_0^1{\frac{u_0^\prime(\alpha)^2d\alpha}{\mathcal{J}(\alpha,t)^{2+\frac{1}{\lambda}}}}-2\lambda\left(\int_0^1{\frac{u_0^\prime(\alpha)\,d\alpha}{\mathcal{J}(\alpha,t)^{1+\frac{1}{\lambda}}}}\right)^2\right)\bar{\mathcal{K}}_0(t)^{-4\lambda}>0
\end{split}
\end{equation*}
for $\lambda\in[-1,0)$ and $\eta\in(0,\eta_*).$ This implies
\begin{equation}
\label{eq:almost}
\begin{split}
\dot{\bar{\mathcal{K}}}_0(t)=\bar{\mathcal{K}}_0(t)^{-2\lambda}\int_0^1{\frac{u_0^{\prime}(\alpha)d\alpha}{\mathcal{J}(\alpha,t)^{1+\frac{1}{\lambda}}}}>0.
\end{split}
\end{equation}
Then, using (\ref{eq:almost}), $\bar{\mathcal{K}}_0(0)=1$ and $m_0\leq u_0^\prime(\alpha)\leq M_0$ yield (\ref{eq:quick3}). Similarly, one can show that 
\begin{equation}
\label{eq:quick4}
\begin{split}
1\leq\bar{\mathcal{K}}_1(t)\leq\left(\frac{m_0}{m_0-M_0}\right)^{1+\frac{1}{\lambda}}.
\end{split}
\end{equation}
Consequently, (\ref{eq:maxp})ii), (\ref{eq:mainsolu}), (\ref{eq:quick3}) and (\ref{eq:quick4}) imply that $$m(t)=u_x(\gamma(\underline\alpha_j,t),t)\to-\infty$$ as $\eta\uparrow\eta_*.$ On the other hand, by (\ref{eq:quick3}), (\ref{eq:quick4}) and the definition of $m_0,$ we find that $u_x(\gamma(\alpha,t),t)$  remains bounded for all $\alpha\neq\underline\alpha_j$ as $\eta\uparrow\eta_*.$ The existence of a finite blow-up time $t_*>0$ is guaranteed by (\ref{eq:etaivp}) and (\ref{eq:quick3}). Although $t_*$ can be computed explicitly from (\ref{eq:assympttwo}), (\ref{eq:quick3}) provides the simple estimate\footnote[5]{Which we may contrast to (\ref{eq:lowt}). Notice that (\ref{eq:assympttwo}) implies that the two cases coincide ($t_*=\eta_*$) in the case of Burgers' equation $\lambda=-1.$} 
\begin{equation}
\label{eq:esttime2}
\begin{split}
\eta_*\left(1-\frac{M_0}{m_0}\right)^{-2}\leq t_*\leq\eta_*.
%\frac{m_0}{\lambda(m_0-M_0)^2}\leq t_*\leq\eta_*.
\end{split}
\end{equation}
Also, since the maximum $M(t)$ remains finite as $t\uparrow t_*,$ setting $\alpha=\overline\alpha_i$ in (\ref{eq:mainsolu}) and using (\ref{eq:maxp})i) and (\ref{eq:ref}) gives $\dot M(t)<\lambda(M(t))^2<0,$ which implies $$0<M(t)\leq M_0$$ for all $t\in[0,t_*]$ and $\lambda\in[-1,0).$ Finally, (\ref{eq:sumjacblowalpha-10}) follows directly from (\ref{eq:sum}), (\ref{eq:quick3}) and the definition of $m_0.$ See \S \ref{subsec:smooth} for examples.\qquad\end{proof}

\subsubsection{\textbf{Further $L^p$ Regularity}}\hfill
\label{subsubsec:lpregsmooth}

In this section, we prove Theorem \ref{thm:lpintro}. In particular, we will see how the two-sided, everywhere blow-up (or one-sided, discrete blow-up) found in theorems \ref{thm:theoremtwogen}, \ref{thm:theoremfour} and \ref{thm:theoremthree}, can be associated with stronger (or weaker) $L^p$ regularity. Before proving the Theorem, we derive basic upper and lower bounds for $\left\|u_x\right\|_{p},\, p\in[1,+\infty),$ as well as write down explicit formulas for the energy function $E(t)=\left\|u_x\right\|_2^2\,$ and derivative $ \dot E(t)$, and estimate the blow-up rates of relevant time-dependent integrals. From (\ref{eq:sum}) and (\ref{eq:mainsolu}),
\begin{equation}
\label{eq:lpreg1}
\left|u_x(\gamma(\alpha,t),t)\right|^p\gamma_\alpha(\alpha,t)=\frac{\left|f(\alpha,t)\right|^p}{\left|\lambda\eta(t)\right|^p\bar{\mathcal{K}}_0(t)^{^{1+2\lambda p}}}
\end{equation}
for $t\in[0,t_*),\, p\in[1,+\infty),\, \lambda\neq 0$ and $$f(\alpha,t)=\frac{1}{\mathcal{J}(\alpha,t)^{^{1+\frac{1}{\lambda p }}}}-\frac{\bar{\mathcal{K}}_1(t)}{\bar{\mathcal{K}}_0(t)\mathcal{J}(\alpha,t)^{\frac{1}{\lambda p}}}.$$ Integrating (\ref{eq:lpreg1}) in $\alpha$ and using periodicity then gives
\begin{equation}
\label{eq:lpreg2}
\begin{split}
\left\|u_x(x,t)\right\|_p^p=\frac{1}{\left|\lambda\eta(t)\right|^p\bar{\mathcal{K}}_0(t)^{^{1+2\lambda p}}}\int_0^1{\left|f(\alpha,t)\right|^pd\alpha}.
\end{split}
\end{equation}
In particular, setting $p=2$ yields the following formula for the energy $E(t):$
\begin{equation}
\label{eq:energy2}
\begin{split}
E(t)=\left(\lambda\eta(t)\bar{\mathcal{K}}_0(t)^{1+2\lambda}\right)^{-2}\left(\bar{\mathcal{K}}_0(t)\bar{\mathcal{K}}_2(t)-\bar{\mathcal{K}}_1(t)^2\right).
\end{split}
\end{equation}
Furthermore, multiplying (\ref{eq:nonhomo})i) by $u_x,$ integrating by parts and using (\ref{eq:pbc}), (\ref{eq:sum}) and (\ref{eq:mainsolu}) gives, after some simplification,
\begin{equation}
\label{eq:dotderenergy}
\begin{split}
\dot E(t)&=(1+2\lambda)\int_0^1{u_x(x,t)^3dx}
\\
&=(1+2\lambda)\int_0^1{(u_x(\gamma(\alpha,t),t))^3\gamma_\alpha(\alpha,t)\,d\alpha}
\\
&=\frac{1+2\lambda}{(\lambda\eta(t))^3}
\left[\frac{\bar{\mathcal{K}}_3(t)}{\bar{\mathcal{K}}_1(t)}-\frac{3\bar{\mathcal{K}}_2(t)}{\bar{\mathcal{K}}_0(t)}+2\left(\frac{\bar{\mathcal{K}}_1(t)}{\bar{\mathcal{K}}_0(t)}\right)^2\right]\frac{\bar{\mathcal{K}}_1(t)}{\bar{\mathcal{K}}_0(t)^{1+6\lambda}}.
\end{split}
\end{equation}
Now $\bar{\mathcal{K}}_i(t),\, \mathcal{J}(\alpha,t)>0$ for $\eta\in[0,\eta_*)$ (i.e. $t\in[0,t_*)$) and $\alpha\in[0,1].$ As a result 
\begin{equation*}
\begin{split}
\left|f(\alpha,t)\right|^p\leq2^{p-1}\left(\frac{1}{\mathcal{J}(\alpha,t)^{^{p+\frac{1}{\lambda }}}}+\frac{\bar{\mathcal{K}}_1(t)^{^p}}{\bar{\mathcal{K}}_0(t)^{^p}\mathcal{J}(\alpha,t)^{\frac{1}{\lambda}}}\right)
\end{split}
\end{equation*}
which  can be used together with (\ref{eq:lpreg1}) to obtain, upon integration, the upper bound
\begin{equation}
\label{eq:upper}
\begin{split}
\left\|u_x(x,t)\right\|_p^p\leq\frac{2^{p-1}}{\left|\lambda\eta(t)\right|^p\bar{\mathcal{K}}_0(t)^{^{1+2\lambda p}}}\left(\int_0^1{\frac{d\alpha}{\mathcal{J}(\alpha,t)^{^{p+\frac{1}{\lambda }}}}}+\frac{\bar{\mathcal{K}}_1(t)^{^p}}{\bar{\mathcal{K}}_0(t)^{^{p-1}}}\right)
\end{split}
\end{equation}
valid for $t\in[0,t_*),\, p\in[1,+\infty)$ and $\lambda\neq0.$ For a lower bound, notice that by Jensen's inequality,
\begin{equation*}
\begin{split}
\int_0^1{\left|f(\alpha,t)\right|^pd\alpha}\geq\left|\int_0^1{f(\alpha,t)d\alpha}\right|^p
\end{split}
\end{equation*}
for $p\in[1,+\infty).$ Using the above in (\ref{eq:lpreg2}), we find
\begin{equation}
\label{eq:lower}
\begin{split}
\left\|u_x(x,t)\right\|_p\geq\frac{1}{\left|\lambda\eta(t)\right|\bar{\mathcal{K}}_0(t)^{^{2\lambda+\frac{1}{p}}}}\left|\int_0^1{\frac{d\alpha}{\mathcal{J}(\alpha,t)^{^{1+\frac{1}{\lambda p }}}}}-\frac{\bar{\mathcal{K}}_1(t)}{\bar{\mathcal{K}}_0(t)}\int_0^1{\frac{d\alpha}{\mathcal{J}(\alpha,t)^{\frac{1}{\lambda p}}}}\right|.
\end{split}
\end{equation}
Although the right-hand side of (\ref{eq:lower}) is identically zero for $p=1,$ it does allow for the study of $L^p$ regularity of solutions when $p\in(1,+\infty)$\footnote[6]{Also, for $p\in(1,+\infty),$ (\ref{eq:lower}) makes sense as $t\downarrow0$ due to the periodicity of $u_0'.$}. Before proving Theorem \ref{thm:lpintro}, we need to determine any blow-up rates for the appropriate integrals in (\ref{eq:energy2})-(\ref{eq:lower}). By following the argument in theorems \ref{thm:theoremtwogen} and \ref{thm:theoremfour}, we go through the derivation of estimates for $\int_0^1{\mathcal{J}(\alpha,t)^{-\left(1+\frac{1}{\lambda p}\right)}d\alpha}$ with $\lambda\in(1,+\infty),\, p\in[1+\infty)$ and $\eta_*=\frac{1}{\lambda M_0},$ whereas those for $$\int_0^1{\frac{d\alpha}{\mathcal{J}(\alpha,t)^{\frac{1}{\lambda p}}}},\,\,\,\,\,\,\,\,\,\,\,\,\,\int_0^1{\frac{d\alpha}{\mathcal{J}(\alpha,t)^{p+\frac{1}{\lambda}}}}$$ follow similarly and will be simply stated here. For simplicity, assume $u_0'$ attains its maximum value $M_0>0$ at a single $\overline\alpha\in(0,1).$ As before, we consider the case where, near $\overline\alpha,\, u_0'$ has non-vanishing second-order derivative. Accordingly, there is $s \in (0, 1]$ such that $u_0'(\alpha)\sim M_0+C_1(\alpha-\overline\alpha)^2$ for $0\leq\left|\alpha-\overline\alpha\right|\leq s$ and $C_1=u_0'''(\overline\alpha)/2<0.$ Then $\epsilon-u_0'(\alpha)+M_0\sim\epsilon-C_1(\alpha-\overline\alpha)^2$
for $\epsilon>0.$ Given $\lambda>1$ and $p\geq1$, we let $b=1+\frac{1}{\lambda p}$ in Lemma \ref{prop:prop} to obtain
\begin{equation}
\label{eq:lastreal}
\begin{split}
\int_{\overline\alpha-s}^{\overline\alpha+s}{\frac{d\alpha}{(\epsilon-u_0'(\alpha)+M_0)^{b}}}\sim\int_{\overline\alpha-s}^{\overline\alpha+s}{\frac{d\alpha}{(\epsilon-C_1(\alpha-\overline\alpha)^2)^{b}}}=\frac{2s}{\epsilon^{b}}\,{}_2F_1\left[\frac{1}{2},b;\frac{3}{2};\frac{C_1s^2}{\epsilon}\right]
\end{split}
\end{equation}
for $\epsilon\geq -C_1\geq-s^2C_1>0.$ Now, if we let $\epsilon>0$ become small enough, so that eventually $\frac{C_1s^2}{\epsilon}<-1$, proposition \ref{lem:analcont} implies
\begin{equation*}
\begin{split}
\frac{2s}{\epsilon^b}{}_2F_1\left[\frac{1}{2},b;\frac{3}{2};\frac{C_1s^2}{\epsilon}\right]=\frac{2s}{(1-2b)(-s^2C_1)^b}+\frac{\Gamma\left(b-\frac{1}{2}\right)}{\Gamma(b)}\sqrt{-\frac{\pi}{C_1}}\,\epsilon^{\frac{1}{2}-b}+\zeta(\epsilon)
\end{split}
\end{equation*}
for $\lambda\neq2/p$, and $\zeta(\epsilon)=\textsl{o}(1)$ as $\epsilon\to0.$ Using the above on (\ref{eq:lastreal}) yields
\begin{equation}
\label{eq:again}
\begin{split}
\int_{0}^{1}{\frac{d\alpha}{(\epsilon-u_0'(\alpha)+M_0)^{b}}}\sim\frac{\Gamma(b-1/2)}{\Gamma(b)}\sqrt{-\frac{\pi}{C_1}}\,\epsilon^{\frac{1}{2}-b}
\end{split}
\end{equation}
for $\epsilon>0$ small. Then, setting $\epsilon=\frac{1}{\lambda\eta}-M_0$ into (\ref{eq:again}) gives
\begin{equation}
\label{eq:estgenp1}
\int_0^1{\frac{d\alpha}{\mathcal{J}(\alpha,t)^{1+\frac{1}{\lambda p}}}}\sim C\mathcal{J}(\overline\alpha,t)^{-(\frac{1}{2}+\frac{1}{\lambda p})}
\end{equation}
for $\eta_*-\eta>0$ small, $\eta_*=\frac{1}{\lambda M_0}$, $p\in[1,+\infty)$ and $\lambda\in(1,+\infty)$\footnote[7]{When $\lambda=2/p,\, b=3/2$ and (\ref{eq:estgenp1}) reduces to (\ref{eq:genalpha2})ii).}. For the other cases and remaining integrals, we follow a similar argument to find
\begin{equation}
\label{eq:estgenp1neg}
\int_0^1{\frac{d\alpha}{\mathcal{J}(\alpha,t)^{1+\frac{1}{\lambda p}}}}\sim \frac{C}{\mathcal{J}(\underline\alpha,t)^{\frac{1}{2}+\frac{1}{\lambda p}}},\,\,\,\,\,\,\,\,\,\lambda<-\frac{2}{p},\,\,\,\,\,p\in[1,+\infty),
\end{equation}
\begin{equation}
\label{eq:estgenp2}
\int_0^1{\frac{d\alpha}{\mathcal{J}(\alpha,t)^{\frac{1}{\lambda p}}}}\sim
\begin{cases}
C,\,\,\,\,\,&\lambda>\frac{2}{p},\,\,\,p\geq1\,\,\,\,\text{or}\,\,\,\,\lambda\in\mathbb{R}^-,
\\
C\mathcal{J}(\overline\alpha,t)^{\frac{1}{2}-\frac{1}{\lambda p}},\,\,\,&1<\lambda<\frac{2}{p},\,\,\,\,\,\,\,1<p<2
\end{cases}
\end{equation}
and
\begin{equation}
\label{eq:upper1}
\begin{split}
\int_0^1{\frac{d\alpha}{\mathcal{J}(\alpha,t)^{p+\frac{1}{\lambda}}}}\sim C,\,\,\,\,\,\,\,\,\,\,\,\,\frac{2}{1-2p}<\lambda<0,\,\,\,p\geq1
\end{split}
\end{equation}
where the positive constants $C$ may depend on the choices for $\lambda$ and $p.$

Recall from Theorem \ref{thm:theoremtwogen} (see also appendix A) that
\begin{equation}
\label{eq:ah1}
\begin{split}
\lim_{t\to+\infty}\left\|u_x\right\|_{\infty}<+\infty,\,\,\,\,\,\,\,\,\,\,\lambda\in[0,1].
\end{split}
\end{equation}
In contrast, Theorem \ref{thm:maintheorem1}, which we established in Theorems  \ref{thm:theoremtwogen}, \ref{thm:theoremfour} and \ref{thm:theoremthree}, showed the existence of a finite $t_*>0$ such that 
\begin{equation}
\label{eq:ah2}
\begin{split}
\lim_{t\uparrow t_*}\left\|u_x\right\|_{\infty}=+\infty,\,\,\,\,\,\,\,\,\,\,\,\,\,\,\lambda\in\mathbb{R}\backslash[0,1].
\end{split}
\end{equation}
Consequently, $\left\|u_x\right\|_p$ exists globally for all $p\in[1,+\infty]$ and $\lambda\in[0,1]$. In the case of (\ref{eq:ah2}), Theorem \ref{thm:lp} below further examines the $L^p$ regularity of $u_x$ as $t$ approaches the finite $L^{\infty}$ blow-up time $t_*.$

\begin{theorem}
\label{thm:lp}
Consider the initial boundary value problem (\ref{eq:nonhomo})-(\ref{eq:pbc}) and let $t_*>0$ denote the finite $L^{\infty}$ blow-up time in Theorem \ref{thm:maintheorem1}. There exist smooth, mean-zero initial data such that:
\begin{enumerate}
\item\label{it:1lp} For $p\in(1,+\infty)$ and $\lambda\in(-\infty,-2/p]\cup(1,+\infty)$,\,\, $\lim_{t\uparrow t_*}\left\|u_x\right\|_p=+\infty$. 
%For $p>1$ and $\lambda\in(-\infty,-2]\cup(1,+\infty)$,\,\, $\lim_{t\uparrow t_*}\left\|u_x\right\|_p=+\infty$. Similarly for $p\in(1,+\infty)$ and $\lambda\in(-2,-2/p]$.
\item\label{it:2lp} For $p\in[1,+\infty)$ and $\frac{2}{1-2p}<\lambda<0,\,\, \lim_{t\uparrow t_*}\left\|u_x\right\|_p<+\infty.$
\item\label{it:energyitem} The energy $E(t)=\left\|u_x\right\|_2^2$ diverges as $t\uparrow t_*$ if $\lambda\in(-\infty,-2/3]\cup(1,+\infty)$ but remains finite for $t\in[0,t_*]$ when $\lambda\in(-2/3,0)$. Moreover, $\dot E(t)$ blows up to $+\infty$ as $t\uparrow t_*$ if $\lambda\in(-\infty,-1/2)\cup(1,+\infty)$ and $\dot E(t)\equiv0$ for $\lambda=-1/2;$ whereas, $\lim_{t\uparrow t_*}\dot E(t)=-\infty$ when $\lambda\in(-1/2,-2/5]$ but remains bounded, for all $t\in[0,t_*]$, if $\lambda\in(-2/5,0)$. 
\end{enumerate}
\end{theorem}
\begin{proof}
\textit{\textbf{Case $\lambda,\, p\in(1,+\infty).$}}
\vspace{0.02in}

First, consider the lower bound (\ref{eq:lower}) for $p\in(1,2)$ and $\lambda\in(1,2/p).$ Then, $\lambda\in(1,2)$ so that (\ref{eq:alphaposgen7}), (\ref{eq:alphaposgen6})i), (\ref{eq:estgenp1}) and (\ref{eq:estgenp2})ii) imply \begin{equation*}
\begin{split}
\left\|u_x(x,t)\right\|_p^p\geq\frac{\left|\int_0^1{f(\alpha,t)d\alpha}\right|^p}{\left|\lambda\eta(t)\right|^p\bar{\mathcal{K}}_0(t)^{^{1+2\lambda p}}}\sim C\mathcal{J}(\overline\alpha,t)^{\sigma(\lambda,p)}
\end{split}
\end{equation*}
for $\eta_*-\eta>0$ small and $\sigma(\lambda,p)=\frac{3p}{2}-\frac{1}{2}-\lambda p.$ By the above restrictions on $\lambda$ and $p,$ we see that $\sigma(\lambda,p)<0$ for $\frac{1}{2}\left(3-\frac{1}{p}\right)<\lambda<\frac{2}{p}$, $p\in(1,5/3).$ Then, by choosing $p-1>0$ arbitrarily small, $\left\|u_x\right\|_p\to+\infty$ as $t\uparrow t_*$ for $\lambda\in(1,2).$ Next, let $\lambda\in(2,+\infty)$ and $p\in(1,+\infty).$ This means $\lambda>\frac{2}{p}$, and so (\ref{eq:alphaposgen7}), (\ref{eq:alphaposgen6})ii), (\ref{eq:estgenp1}) and (\ref{eq:estgenp2})i) now yield
\begin{equation}
\label{eq:lower3}
\begin{split}
\left\|u_x(x,t)\right\|_p\geq\frac{\left|\int_0^1{f(\alpha,t)d\alpha}\right|}{\left|\lambda\eta(t)\right|\bar{\mathcal{K}}_0(t)^{^{2\lambda+\frac{1}{p}}}}\sim \frac{C}{\mathcal{J}(\overline\alpha,t)^{\frac{1}{2}+\frac{1}{\lambda}}}\to+\infty
\end{split}
\end{equation}
as $t\uparrow t_*.$ This proves part (\ref{it:1lp}) of the Theorem for $\lambda\in(1,+\infty)$.\footnote[8]{If $\lambda=2,\, \lambda>\frac{2}{p}$ for $p>1$ and result follows from (\ref{eq:genalpha2}), (\ref{eq:lower}), (\ref{eq:estgenp1}) and (\ref{eq:estgenp2})i).}
\vspace{0.02in}

\textsl{\textit{\textbf{Case $\lambda\in(-\infty,0)$ and $p\in[1,+\infty)$.}}}
\vspace{0.02in}

For $\lambda\in(-\infty,0)$, we keep in mind the estimates (\ref{eq:newest}), (\ref{eq:lastfin})i), (\ref{eq:quick3}) and (\ref{eq:quick4}) which describe the behaviour of $\bar{\mathcal{K}}_i(t),\, i=0,1$ as $\eta\uparrow\eta_*.$ Consider the upper bound (\ref{eq:upper}) for $p\in[1,+\infty)$ and $\frac{2}{1-2p}<\lambda<0.$ Then $\lambda\in(-2,0)$, equation (\ref{eq:upper1}), and  the aforementioned estimates imply that, as $t\uparrow t_*$, 
\begin{equation*} 
\begin{split}
\left\|u_x(x,t)\right\|_p^p\leq\frac{2^{p-1}}{\left|\lambda\eta(t)\right|^p\bar{\mathcal{K}}_0(t)^{^{1+2\lambda p}}}\left(\int_0^1{\frac{d\alpha}{\mathcal{J}(\alpha,t)^{^{p+\frac{1}{\lambda }}}}}+\frac{\bar{\mathcal{K}}_1(t)^{^p}}{\bar{\mathcal{K}}_0(t)^{^{p-1}}}\right)\to C.
\end{split}
\end{equation*}
Here, $C\in\mathbb{R}^+$ depends on the choice of $\lambda$ and $p.$ By the above, we conclude that 
$$\lim_{t\uparrow t_*}\left\|u_x(x,t)\right\|_p<+\infty$$ 
for $\frac{2}{1-2p}<\lambda<0$ and $p\in[1,+\infty).$ Now, consider the lower bound (\ref{eq:lower}) with $p\in(1,+\infty)$ and $-2<\lambda<-\frac{2}{p}<\frac{2}{1-2p}.$ Then, by (\ref{eq:estgenp1neg}), (\ref{eq:estgenp2})i) and corresponding estimates on $\bar{\mathcal{K}}_i(t),\, i=0,1,$ we find that
\begin{equation}
\label{eq:lower23}
\begin{split}
\left\|u_x(x,t)\right\|_p\geq\frac{\left|\int_0^1{f(\alpha,t)d\alpha}\right|}{\left|\lambda\eta(t)\right|\bar{\mathcal{K}}_0(t)^{^{2\lambda+\frac{1}{p}}}}\sim C\mathcal{J}(\underline\alpha,t)^{-\left(\frac{1}{2}+\frac{1}{\lambda}\right)}
\end{split}
\end{equation}
for $\eta_*-\eta>0$ small. Therefore, 
\begin{equation}
\label{eq:bah1}
\begin{split}
\lim_{t\uparrow t_*}\left\|u_x(x,t)\right\|_p=+\infty
\end{split}
\end{equation}
for $p\in(1,+\infty)$ and $\lambda\in(-2,-2/p]$\footnote[9]{For the case $\lambda=-\frac{2}{p}$ with $p\in(1,+\infty),$ we simply use (\ref{eq:alpha-2}) instead of (\ref{eq:estgenp1neg}).}. 
Finally, let $\lambda\in(-\infty,-2)$ and $p\in(1,+\infty).$ Then $\lambda<-\frac{2}{p}$ and it is easy to check that (\ref{eq:lower23}), with different constants $C>0,$ also holds. As a result, (\ref{eq:bah1}) follows for $p>1$ and $\lambda\in(-\infty,-2]$\footnote[10]{If $\lambda=-2,\, \lambda<-\frac{2}{p}$ for $p>1.$ Result follows as above with (\ref{eq:alpha-2}) instead of (\ref{eq:estgenp1neg}).}. Since we already established that $u_x\in L^{\infty}$ for all time when $\lambda\in[0,1]$ (see Theorem \ref{thm:theoremtwogen}), this concludes the proof of parts (\ref{it:1lp}) and (\ref{it:2lp}) of the Theorem. 

For part (\ref{it:energyitem}), notice that when $p=2,$ parts (\ref{it:1lp}) and (\ref{it:2lp}), as well as Theorem \ref{thm:theoremtwogen} imply that, as $t\uparrow t_*$, both $E(t)=\left\|u_x\right\|_2^2$ and $\dot E(t)$ diverge to $+\infty$ for $\lambda\in(-\infty,-1]\cup(1,+\infty)$ while $E(t)$ remains finite if $\lambda\in(-2/3,1].$ Therefore we still have to establish the behaviour of $E(t)$ when $\lambda\in(-1,-2/3]$ and $\dot E(t)$ for $\lambda\in(-1,0)\backslash\{-1/2\}.$ From (\ref{eq:quick3}), (\ref{eq:quick4}) and (\ref{eq:energy2}), we see that, as $t\uparrow t_*,$ any blow-up in $E(t)$ for $\lambda\in(-1,-2/3]$ must come from the $\bar{\mathcal{K}}_2(t)$ term. Using proposition \ref{lem:analcont} and Lemma \ref{prop:prop}, we estimate\footnote[11]{Under the usual assumption $u_0'''(\underline\alpha)\neq0.$}
\begin{equation}
\label{eq:estthe2}
\bar{\mathcal{K}}_2(t)\sim
\begin{cases}
C\mathcal{J}(\underline\alpha,t)^{-\left(\frac{3}{2}+\frac{1}{\lambda}\right)},\,\,\,\,\,\,\,\,\,\,&\lambda\in(-1,-2/3),
\\
-C\log\left(\mathcal{J}(\underline\alpha,t)\right),\,\,\,\,\,\,\,\,\,\,&\lambda=-2/3,
\\
C,\,\,\,\,\,\,\,&\lambda\in(-2/3,0)
\end{cases}
\end{equation}
for $\eta_*-\eta>0$ small. Then, (\ref{eq:quick3}), (\ref{eq:quick4}), and (\ref{eq:energy2}) imply that, as $t\uparrow t_*,$ both $E(t)$ and $\dot E(t)$ blow-up to $+\infty$ for $\lambda\in(-1,-2/3].$ Now, from (\ref{eq:dotderenergy})i), we have that
\begin{equation}
\label{eq:l3}
\left|\dot E(t)\right|\leq\left|1+2\lambda\right|\left\|u_x\right\|_3^3
\end{equation}
so that Theorem \ref{thm:theoremtwogen} implies that $\dot E(t)$ remains finite for all time if $\lambda\in[0,1]$. Also, since $3+\frac{1}{\lambda}\leq0$ for all $\lambda\in[-1/3,0)$, we use (\ref{eq:quick3}), (\ref{eq:quick4}) and (\ref{eq:estthe2})iii) on (\ref{eq:dotderenergy})iii) to conclude that 
$$\lim_{t\uparrow t_*}\left|\dot E(t)\right|<+\infty$$ 
for $\lambda\in[-1/3,0)$ as well. Moreover, by part (\ref{it:2lp}), $\lim_{t\uparrow t_*}\left\|u_x\right\|_3<+\infty$ for $\lambda\in(-2/5,0).$ Then, (\ref{eq:l3}) implies that $\dot E(t)$ also remains finite for $\lambda\in(-2/5,-1/3)$. Lastly, estimating $\bar{\mathcal{K}}_3(t)$ yields
\begin{equation}
\label{eq:estthe3}
\bar{\mathcal{K}}_3(t)\sim
\begin{cases}
C\mathcal{J}(\underline\alpha,t)^{-\left(\frac{5}{2}+\frac{1}{\lambda}\right)},\,\,\,\,\,\,\,\,\,\,&\lambda\in(-2/3,-2/5),
\\
%C\mathcal{J}(\underline\alpha,t)^{-1},\,\,\,\,\,\,\,\,\,\,\,\,\,\,&\lambda=-\frac{2}{3},
%\\
-C\log\left(\mathcal{J}(\underline\alpha,t)\right),\,\,\,\,\,\,\,\,\,&\lambda=-2/5.
\end{cases}
\end{equation}
As a result, (\ref{eq:quick3}), (\ref{eq:quick4}), (\ref{eq:estthe2})iii) and (\ref{eq:dotderenergy})iii) imply that 
\begin{equation*}
\lim_{t\uparrow t_*}\dot E(t)=
\begin{cases}
+\infty,\,\,\,\,\,\,\,\,\,\,&\lambda\in(-2/3,-1/2),
\\
-\infty,\,\,\,\,\,\,\,\,\,&\lambda\in(-1/2,-2/5].
\end{cases}
\end{equation*}
We refer the reader to table \ref{table:lptable} in \S \ref{sec:summary} for a summary of the above results.\qquad\end{proof}

Notice that Theorems \ref{thm:maintheorem1}, \ref{thm:lp} and inequality (\ref{eq:l3}) yield a complete description of the $L^3$ regularity for $u_x$: if $\lambda\in[0,1]$, $\lim_{t\to+\infty}\left\|u_x\right\|_3=C$ where $C\in\mathbb{R}^+$ for $\lambda=1$ but $C=0$ when $\lambda\in(0,1)$, whereas, for $t_*>0$ the finite $L^{\infty}$ blow-up time for $u_x$ in Theorem \ref{thm:maintheorem1}, 
\begin{equation}
\label{eq:l3final}
\lim_{t\uparrow t_*}\left\|u_x(x,t)\right\|_3=
\begin{cases}
+\infty,\,\,\,\,\,\,\,\,\,\,&\lambda\in(-\infty,-2/5]\cup(1,+\infty),
\\
C,\,\,\,\,\,\,&\lambda\in(-2/5,0)
\end{cases}
\end{equation}
where the positive constants $C$ depend on the choice of $\lambda\in(-2/5,0).$

\begin{remark}
Theorem \ref{thm:lp} implies that for every $p>1$, $L^p$ blow-up occurs for $u_x$ if $\lambda\in\mathbb{R}\backslash(-2,1],$ whereas for $\lambda\in(-2,0)$, $u_x$ remains in $L^1$ but blows up in particular, smaller $L^p$ spaces. This suggests a weaker type of blow-up for the latter which certainly agrees with our $L^{\infty}$ results where a ``stronger'', two-sided, everywhere blow-up takes place for $\lambda\in\mathbb{R}\backslash(-2,1],$ but a ``weaker'', one-sided, discrete blow-up occurs when $\lambda\in(-2,0).$ 
\end{remark}

\begin{remark}
For $V(t)=\int_0^1{u_x^3dx}$, the authors in \cite{Okamoto1} derived a finite upper bound 
\begin{equation}
\label{eq:oktime}
T^*=\left(\frac{3}{\left|1+2\lambda\right|E(0)}\right)^{\frac{1}{2}}
\end{equation}
for the blow-up time of $E(t)$ for $\lambda<-1/2$ and 
\begin{equation}
\label{eq:okcond}
V(0)<0,\,\,\,\,\,\,\,\,\,\,\,\,\,\,\,\,\,\frac{\left|1+2\lambda\right|}{2}V(0)^2\geq\frac{2}{3}E(0)^3.
\end{equation}
If (\ref{eq:okcond})i) holds but we reverse (\ref{eq:okcond})ii), then they proved that $\dot E(t)$ blows up instead. Now, from Theorem \ref{thm:lp}(\ref{it:energyitem}) we have that, in particular for $\lambda\in(-2/3,-1/2)$, $E(t)$ remains bounded for $t\in[0,t_*]$ but $\dot E(t)\to+\infty$ as $t\uparrow t_*$. Here, $t_*>0$ denotes the finite $L^{\infty}$ blow-up time for $u_x$ (see Theorem \ref{thm:theoremthree}) and satisfies (\ref{eq:esttime2}). Therefore, further discussion is required to clarify the apparent discrepancy between the two results for $\lambda\in(-2/3,-1/2)$ and $u_0'$ satisfying both conditions in (\ref{eq:okcond}). Our claim is that for these values of $\lambda$, $t_*<T^*$. Specifically, $E(t)$ remains finite for all $t\in[0,t_*]\subset[0,T^*]$, while $\dot E(t)\to+\infty$ as $t\uparrow t_*$. From (\ref{eq:dotderenergy})i) and (\ref{eq:okcond})ii), we have that $\frac{\dot E(0)^2}{2\left|1+2\lambda\right|}\geq\frac{2}{3}E(0)^3$, or $\frac{1}{(\left|1+2\lambda\right|E(0))^3}\geq\frac{4}{3(\left|1+2\lambda\right|\dot E(0))^2}.$ As a result, (\ref{eq:oktime}) yields 
\begin{equation} 
\label{eq:oktime2}
T^*\geq\left(\frac{6}{\left|1+2\lambda\right|\dot E(0)}\right)^{\frac{1}{3}}
\end{equation}
where we used $\dot E(0)>0$; a consequence of (\ref{eq:dotderenergy})i), (\ref{eq:okcond})i) and $\lambda\in(-2/3,-1/2)$. Now, for instance, suppose $0<M_0\leq\left|m_0\right|$.\footnote[12]{A natural case to consider given (\ref{eq:okcond})i).} Then
\begin{equation} 
\label{eq:simpleineq}
-V(0)=\bigg|\int_0^1{u_0'(x)^3dx}\bigg|\leq\max_{x\in[0,1]}\left|u_0'(x)\right|^3=\left|m_0\right|^3,
\end{equation}
which we use on (\ref{eq:dotderenergy})i) to obtain $0<\dot E(0)\leq\left|1+2\lambda\right|\left|m_0\right|^3,$ or equivalently
\begin{equation} 
\label{eq:p5}
\frac{6}{\left|1+2\lambda\right|\dot E(0)}\geq\frac{6}{\left|1+2\lambda\right|^2\left|m_0\right|^3}.
\end{equation}
Consequently, (\ref{eq:esttime2}), (\ref{eq:oktime2}) and (\ref{eq:p5}) yield 
\begin{equation}
\label{eq:oktime3}
T^*\geq\left(\frac{6}{\left|1+2\lambda\right|^{2}\left|m_0\right|^3}\right)^{\frac{1}{3}}>\frac{1}{\left|1+2\lambda\right|^{\frac{2}{3}}\left|m_0\right|}>\frac{1}{\left|\lambda\right|\left|m_0\right|}=\eta_*\geq t_* 
\end{equation}
for $\lambda\in(-2/3,-1/2)$. If $\lambda\leq-2/3$, both results concerning $L^2$ blow-up of $u_x$ coincide. Furthermore, in \cite{Wunsch3} the authors derived a finite upper bound $T_*=\frac{3}{(1+3\lambda)}V(0)^{-\frac{1}{3}}$ for the blow-up time of $V(t)$ to negative infinity valid as long as $V(0)<0$ and $\lambda<-1/3$. Clearly, $T_*$ also serves as an upper bound for the breakdown of $\left\|u_x\right\|_3$ for $\lambda<-1/3$, or $\dot E(t)=(1+2\lambda)V((t)$ if $\lambda\in(-\infty,-1/3)\backslash\{-1/2\}$. However, (\ref{eq:l3final}) and Theorem \ref{thm:lp}(\ref{it:1lp}) prove the existence of a finite $t_*>0$ such that, particularly for $\lambda\in(-2/5,-1/3]$, $\left\|u_x\right\|_3$ remains finite for $t\in[0,t_*]$ while $\lim_{t\uparrow t_*}\left\|u_x\right\|_{6}=+\infty$. This in turn implies the local boundedness of $\dot E(t)$ for $t\in[0,t_*]$ and $\lambda\in(-2/5,-1/3]$. Similar to the previous case, we claim that $t_*<T_*$. Here, once again, we consider the case $0<M_0\leq\left|m_0\right|$. Accordingly, (\ref{eq:esttime2}) and (\ref{eq:simpleineq}) imply 
$$T_*=\frac{3}{(1+3\lambda)V(0)^{\frac{1}{3}}}\geq\frac{3}{\left|1+3\lambda\right|\left|m_0\right|}>\frac{1}{\left|\lambda\right|\left|m_0\right|}=\eta_*\geq t_*.$$
For the remaining values $\lambda\leq-2/5$, both our results and those established in \cite{Wunsch3} regarding blow-up of $V(t)$ agree. A simple example is given by $u_0'(x)=\sin(2\pi x)+\cos(4\pi x)$ for which $V(0)=-3/4$, $E(0)=1$, $m_0=-2$ and $M_0\sim1.125$. Then, for $\lambda=-3/5\in(-2/3,-1/2)$, we have $T^*=\sqrt{15}>\eta_*=5/6\geq t_*\geq0.34$, whereas, if $\lambda=-7/20\in(-2/5,-1/3)$, $T_*=20(6)^{2/3}>10/7=\eta_*\geq t_*\geq0.59$.
%data satisfying certain integral conditions, $E(t)$ blows up in finite-time for $\lambda\in(-\infty,-1/2),$ otherwise $\dot E(t)$ could blow-up instead. However, our results indicate that is the latter which describes the actual blow-up mechanism for our class of  smooth data when $\lambda\in(-2/3,-1/2),$ i.e. if $\lambda\in(-2/3,-1/2),$ $u_x$ blows up  in the $L^3$ rather than the $L^2$ norm ($\dot E(t)\to+\infty$), whereas for $\lambda\in(-\infty,-2/3],\, \left\|u_x\right\|_2$ itself diverges.
\end{remark}

\subsection{\textbf{Piecewise Constant and Piecewise Linear Initial Data}}
\label{subsec:pcpl}
In the previous section, we took smooth data $u_0^\prime$ which attained its extreme values $M_0>0>m_0$ at finitely many points $\overline\alpha_i$ and $\underline\alpha_j\in [0,1]$, respectively,  with $u_0'$ having, relative to the sign of $\lambda,$ quadratic local behaviour near these locations. In this section, two other classes of data are considered which violate these assumptions. In \S \ref{subsubsec:pc}, $L^p$ regularity of solutions  is examined for $u_0^\prime(\alpha)\in PC_{\mathbb{R}}(0,1)$, the class of mean-zero, piecewise constant functions. Subsequently, the case $u_0^{\prime\prime}(\alpha)\in PC_{\mathbb{R}}(0,1)$ is examined via a simple example in \S \ref{subsubsec:pl}.

\subsubsection{\textbf{$n-$phase Piecewise Constant $u_0^\prime(x)$}}\hfill
\label{subsubsec:pc}

Let $\chi_i(\alpha),\,i=1,...,n$ denote the characteristic function for the intervals $\Omega_i=(\alpha_{i-1},\alpha_i)\subset[0,1]$ with $\alpha_0=0,\, \alpha_n=1$ and $\Omega_j\cap\Omega_k=\varnothing,\, j\neq k,$ i.e.
\begin{equation}
\label{eq:characteristic}
\chi_i(\alpha)=
\begin{cases}
1,\,\,\,\,\,\,\,\,\,&\alpha\in\Omega_i,
\\
0,\,\,\,\,\,\,\,\,&\alpha\notin\Omega_i.
\end{cases}
\end{equation}
Then, for $h_i\in\mathbb{R},$ let $PC_{\mathbb{R}}(0,1)$ denote the space of mean-zero, simple functions:
\begin{equation}
\label{eq:pcdata0}
\left\{g(\alpha)\in C^0(0,1)\,\, a.e.\,\bigg|\,g(\alpha)=\sum_{i=1}^nh_i\chi_i(\alpha)\,\,\text{and}\,\,\sum_{i=1}^nh_i\mu(\Omega_i)=0\right\}
\end{equation}
where $\mu(\Omega_i)=\alpha_i-\alpha_{i-1},$ the Lebesgue measure of $\Omega_i.$ Observe that for $u_0^\prime(\alpha)\in PC_{\mathbb{R}}(0,1)$ and $\lambda\neq0,$ (\ref{eq:def}), (\ref{eq:characteristic}) and (\ref{eq:pcdata0}) imply that
\begin{equation}
\label{eq:tone}
\begin{split}
\bar{\mathcal{K}}_i(t)=\sum_{j=1}^n(1-\lambda\eta(t)h_j)^{-i-\frac{1}{\lambda}}\mu(\Omega_j).
\end{split}
\end{equation}
We prove the following Theorem:
\begin{theorem}
\label{thm:pc}
Consider the initial boundary value problem (\ref{eq:nonhomo})-(\ref{eq:pbc}) for periodic $u_0^\prime(\alpha)\in PC_{\mathbb{R}}(0,1).$ Let $T>0$ and assume solutions are defined for all $t\in[0,T].$ Then, the representation formula (\ref{eq:mainsolu}) implies that no global $W^{1,\infty}(0,1)$ solution can exist if $T\geq t_*$, where $t_*=+\infty$ for $\lambda\in[0,+\infty)$ and $0<t_*<+\infty$ otherwise. In addition, $\lim_{t\uparrow t_*}\left\|u_x(x,t)\right\|_1=+\infty$ if $\lambda\in(-\infty,-1),$ while 
\begin{equation*}
\lim_{t\uparrow t_*}\left\|u_x(x,t)\right\|_p=
\begin{cases}
C,\,\,\,\,\,\,\,\,&-\frac{1}{p}\leq\lambda<0,
\\
+\infty,\,\,\,\,\,\,\,\,\,&-1\leq\lambda<-\frac{1}{p}
\end{cases}
\end{equation*}
for $p\geq1,\, \lambda\in[-1,0)$ and $C\in\mathbb{R}^+$ that depend on the choice of $\lambda$ and $p.$
\end{theorem}
\begin{proof}
Let $C$ denote a generic constant which may depend on $\lambda$ and $p.$ Since 
\begin{equation}
\label{eq:pcdata}
u_0^\prime(\alpha)=\sum_{i=1}^nh_i\chi_i(\alpha),
\end{equation}
for $h_i\in\mathbb{R}$ as in (\ref{eq:pcdata0}), then (\ref{eq:sum}) and (\ref{eq:tone}) give
\begin{equation}
\label{eq:sumpc2}
\begin{split}
\gamma_{\alpha}(\alpha,t)^{-\lambda}=(1-\lambda\eta(t)\sum_{i=1}^nh_i\chi_i(\alpha))\left(\sum_{i=1}^n(1-\lambda\eta(t)h_i)^{-\frac{1}{\lambda}}\mu(\Omega_i)\right)^{\lambda}
\end{split}
\end{equation}
for $\eta\in[0,\eta_*),\, \eta_*$ as defined in (\ref{eq:defeta*}) and
\begin{equation}
\label{eq:maxmin}
\begin{cases}
M_0=\max_{i}h_i>0,
\\
m_0=\min_{i}h_i<0.
\end{cases}
\end{equation}
Let $\mathcal{I}_{max}$ and $\mathcal{I}_{min}$ denote the sets of indexes for the intervals $\overline\Omega_i$ and $\underline\Omega_i$ respectively, defined by $\overline\Omega_i\equiv\left\{\overline\alpha\in[0,1]\,|\,u_0^\prime(\overline\alpha)=M_0\right\}$ and $\underline\Omega_i\equiv\left\{\underline\alpha\in[0,1]\,|\,u_0^\prime(\underline\alpha)=m_0\right\}$.
\vspace{0.02in}

\textbf{\textsl{\textit{Global estimates for $\lambda\in(0,+\infty).$}}}
\vspace{0.02in}

Let $\lambda\in(0,+\infty)$ and $\eta_*=\frac{1}{\lambda M_0}.$ Using the above definitions, we may write
\begin{equation}
\label{eq:sumpc3}
\begin{split}
1-\lambda\eta(t)\sum_{i=1}^nh_i\chi_i(\alpha)=1-\lambda\eta(t)\left(\sum_{i\in\mathcal{I}_{max}}M_0\chi_i(\alpha)+\sum_{i\notin\mathcal{I}_{max}}h_i\chi_i(\alpha)\right)
\end{split}
\end{equation}
and
\begin{equation}
\label{eq:sumpc5}
\begin{split}
\sum_{i=1}^n(1-\lambda\eta(t)h_i)^{-\frac{1}{\lambda}}\mu(\Omega_i)=&\sum_{i\in\mathcal{I}_{max}}(1-\lambda\eta(t)M_0)^{-\frac{1}{\lambda}}\mu(\overline\Omega_i)
\\
+&\sum_{i\notin\mathcal{I}_{max}}(1-\lambda\eta(t)h_i)^{-\frac{1}{\lambda}}\mu(\Omega_i).
\end{split}
\end{equation}
Then, for fixed $i\in\mathcal{I}_{max}$ choosing $\overline\alpha\in\overline\Omega_i$ and substituting into (\ref{eq:sumpc3}), we find
\begin{equation}
\label{eq:sumpc4}
\begin{split}
1-\lambda\eta(t)\sum_{i=1}^nh_i\chi_i(\overline\alpha)=1-\lambda\eta(t)M_0.
\end{split}
\end{equation}
Using (\ref{eq:sumpc5}), (\ref{eq:sumpc4}) and (\ref{eq:sumpc2}), we see that, for $\eta\in[0,\eta_*),$  
\begin{equation}
\label{eq:sumpc6}
\begin{split}
\gamma_{\alpha}(\overline\alpha,t)=\left[\sum_{i\in\mathcal{I}_{max}}\mu(\overline\Omega_i)+(1-\lambda\eta(t)M_0)^{\frac{1}{\lambda}}\sum_{i\notin\mathcal{I}_{max}}(1-\lambda\eta(t)h_i)^{-\frac{1}{\lambda}}\mu(\Omega_i)\right]^{-1}.
\end{split}
\end{equation}
Since $1-\lambda\eta(t)u_0^\prime(\alpha)>0$ for all $\eta\in[0,\eta_*)$ and $\alpha\in[0,1],$ (\ref{eq:sumpc6}) implies
%\begin{equation}
%\label{eq:good}
%\begin{split}
%\lim_{\eta\uparrow\eta_*}\sum_{i\notin\mathcal{I}_{max}}(1-\lambda\eta h_i)^{-\frac{1}{\lambda}}\mu(\Omega_i)=C\in\mathbb{R}^+
%\end{split}
%\end{equation}
\begin{equation}
\label{eq:sumpc7}
\begin{split}
\lim_{t\uparrow t_*}{\gamma_{\alpha}(\overline\alpha,t)}=\left(\sum_{i\in\mathcal{I}_{max}}\mu(\overline\Omega_i)\right)^{-1}>0
\end{split}
\end{equation}
for some $t_*>0.$ However, (\ref{eq:sum}), (\ref{eq:etaivp}) and (\ref{eq:pcdata}) give
\begin{equation}
\label{eq:time2}
\begin{split}
dt=\left(1-\lambda\eta(t)\sum_{i=1}^nh_i\chi_i(\alpha)\right)^{-2}\gamma_{\alpha}(\alpha,t)^{^{-2\lambda}}d\eta
\end{split}
\end{equation}
and so, for $\eta_*-\eta>0$ small, (\ref{eq:sumpc2}), (\ref{eq:sumpc5}) and the above observation on $1-\lambda\eta(t)u_0^\prime(\alpha)$ yield, after integration, $t_*-t\sim C\int_\eta^{\eta_*}{(1-\lambda M_0\sigma)^{-2}d\sigma}.$ Consequently, $t_*=+\infty.$ Finally, (\ref{eq:maxp})i), (\ref{eq:alt}) and (\ref{eq:sumpc7}) yield $$\lim_{t\to+\infty}\int_0^t{M(s)\,ds}=-\ln\left(\sum_{i\in\mathcal{I}_{max}}\mu(\overline\Omega_i)\right)>0.$$ If $\alpha=\widetilde\alpha\in\Omega_i$ for some index $i\notin\mathcal{I}_{max},$ so that $1-\lambda\eta(t)u_0'(\widetilde\alpha)=1-\lambda\eta(t)\tilde h$ for $\tilde h<M_0,$ then (\ref{eq:sumpc2}) implies $\gamma_{\alpha}(\widetilde\alpha,t)\sim C(1-\lambda\eta(t)M_0)^{\frac{1}{\lambda}}\to0$ as $t\to+\infty.$ Thus, by (\ref{eq:alt}), we obtain
\begin{equation*}
\begin{split}
\lim_{t\to+\infty}\int_0^t{u_x(\gamma(\widetilde\alpha,s),s)\,ds}=-\infty.
\end{split}
\end{equation*}
We refer to appendix A for the case $\lambda=0.$
\vspace{0.02in}

\textbf{\textsl{\textit{$L^p$ regularity for $p\in[1,+\infty]$ and $\lambda\in(-\infty,0).$}}}
\vspace{0.02in}

Suppose $\lambda\in(-\infty,0)$ so that $\eta_*=\frac{1}{\lambda m_0}.$ We now write
\begin{equation}
\label{eq:sumpc32}
\begin{split}
1-\lambda\eta(t)\sum_{i=1}^nh_i\chi_i(\alpha)=1-\lambda\eta(t)\left(\sum_{i\in\mathcal{I}_{min}}m_0\chi_i(\alpha)+\sum_{i\notin\mathcal{I}_{min}}h_i\chi_i(\alpha)\right)
\end{split}
\end{equation}
and
\begin{equation}
\label{eq:sumpc52}
\begin{split}
\sum_{i=1}^n(1-\lambda\eta(t)h_i)^{\frac{1}{\left|\lambda\right|}}\mu(\Omega_i)=&\sum_{i\in\mathcal{I}_{min}}(1-\lambda\eta(t)m_0)^{\frac{1}{\left|\lambda\right|}}\mu(\underline\Omega_i)
\\
+&\sum_{i\notin\mathcal{I}_{min}}(1-\lambda\eta(t)h_i)^{\frac{1}{\left|\lambda\right|}}\mu(\Omega_i).
\end{split}
\end{equation}
Choose $\underline\alpha\in\underline\Omega_i$ for some $i\in\mathcal{I}_{min}$ and substitute into (\ref{eq:sumpc32}) to obtain
\begin{equation}
\label{eq:sumpc42}
\begin{split}
1-\lambda\eta(t)\sum_{i=1}^nh_i\chi_i(\underline\alpha)=1-\lambda\eta(t)m_0.
\end{split}
\end{equation}
Using (\ref{eq:sumpc52}) and (\ref{eq:sumpc42}) with (\ref{eq:sumpc2}) gives 
\begin{equation}
\label{eq:sumpc62}
\begin{split}
\gamma_{\alpha}(\underline\alpha,t)=\left[\sum_{i\in\mathcal{I}_{min}}\mu(\underline\Omega_i)+\frac{\sum_{i\notin\mathcal{I}_{min}}(1-\lambda\eta(t)h_i)^{\frac{1}{\left|\lambda\right|}}\mu(\Omega_i)}{(1-\lambda\eta(t)m_0)^{\frac{1}{\left|\lambda\right|}}}\right]^{-1}
\end{split}
\end{equation}
for $\eta\in[0,\eta_*).$ Since $1-\lambda\eta(t)u_0^\prime(\alpha)>0$ for $\eta\in[0,\eta_*)$, $\alpha\in[0,1]$ and $\lambda<0,$ we have that $\lim_{t\uparrow t_*}\gamma_{\alpha}(\underline\alpha,t)=0$ or, equivalently by (\ref{eq:maxp})ii) and (\ref{eq:alt}),
\begin{equation*}
\begin{split}
\lim_{t\to t_*}\int_0^t{m(s)ds}=-\infty.
\end{split}
\end{equation*}
The blow-up time $t_*>0$ is now finite. Indeed, (\ref{eq:sumpc2}), (\ref{eq:time2}) and (\ref{eq:sumpc52}) yield the estimate $dt\sim\left(\sum_{i\notin\mathcal{I}_{min}}(1-\lambda\eta(t)h_i)^{\frac{1}{\left|\lambda\right|}}\mu(\Omega_i)\right)^{2\lambda}d\eta$ for $\eta_*-\eta>0$ small and $\lambda<0.$ Since $h_i>m_0$ for any $i\notin\mathcal{I}_{min},$ integrating the latter implies a finite $t_*>0.$ Now, if $\alpha=\alpha^\prime\in\Omega_i$ for some $i\notin\mathcal{I}_{min},$ then $u_0^\prime(\alpha^\prime)=h^\prime$ for $h^\prime>m_0.$ Following the argument in the $\lambda>0$ case yields $$\gamma_\alpha(\alpha^\prime,t)=\left(\sum_{i\notin\mathcal{I}_{min}}(1-\lambda\eta(t)h_i)^{\frac{1}{\left|\lambda\right|}}\mu(\Omega_i)\right)^{-1}(1-\lambda\eta(t)h^\prime)^{\frac{1}{\left|\lambda\right|}},$$ consequently $\lim_{t\uparrow t_*}\gamma_\alpha(\alpha^\prime,t)=C\in\mathbb{R}^+$ and so, by (\ref{eq:alt}), $\int_0^t{u_x(\gamma,s)\,ds}$ remains finite as $t\uparrow t_*$ for every $\alpha^\prime\neq\underline\alpha$ and $\lambda\in(-\infty,0).$
%\begin{equation*}
%\begin{split}
%&\bigg|\lim_{t\uparrow t_*}\int_0^t{u_x(x,s)\circ\gamma(\alpha^\prime,s)\,ds}\bigg|
%\\
%&\,\,\,\,\,\,\,\,\,\,\,\,\,\,\,\,\,\,\,=\bigg|\frac{1}{\left|\lambda\right|}\ln\left(1-\frac{h^\prime}{m_0}\right)-\ln\left(\sum_{i\notin\mathcal{I}_{min}}\left(1-\frac{h_i}{m_0}\right)^{\frac{1}{\left|\lambda\right|}}\mu(\Omega_i)\right)\bigg|<+\infty
%\end{split}
%\end{equation*}

Lastly, we look at $L^p$ regularity of $u_x$ for $p\in[1,+\infty)$ and $\lambda\in(-\infty,0).$ From (\ref{eq:sum}) and (\ref{eq:mainsolu}), 
\begin{equation*}
\begin{split}
\left|u_x(\gamma(\alpha,t),t)\right|^p\gamma_\alpha(\alpha,t)=\frac{\mathcal{K}_0(\alpha,t)|\mathcal{J}(\alpha,t)^{-1}-\bar{\mathcal{K}}_0(t)^{-1}\bar{\mathcal{K}}_1(t)|^p}{\left|\lambda\eta(t)\right|^p\bar{\mathcal{K}}_0(t)^{2\lambda p+1}}
\end{split}
\end{equation*}
for $t\in[0,t_*)$ and $p\in\mathbb{R}.$ Then, integrating in $\alpha$ and using (\ref{eq:tone}) gives
\begin{equation*}
\begin{split}
&\left\|u_x(x,t)\right\|_p^p=\frac{1}{\left|\lambda\eta(t)\right|^p}\left(\sum_{i=1}^n(1-\lambda\eta(t) h_i)^{-\frac{1}{\lambda}}\mu(\Omega_i)\right)^{-(2\lambda p+1)}
\\
&\sum_{j=1}^n\left\{(1-\lambda\eta(t)h_j)^{-\frac{1}{\lambda}}\bigg|(1-\lambda\eta(t)h_j)^{-1}-\frac{\sum_{i=1}^n(1-\lambda\eta(t)h_i)^{-1-\frac{1}{\lambda}}\mu(\Omega_i)}{\sum_{i=1}^n(1-\lambda\eta(t)h_i)^{-\frac{1}{\lambda}}\mu(\Omega_i)}\bigg|^p\mu(\Omega_j)\right\}
\end{split}
\end{equation*}
for $p\in[1,+\infty).$ Splitting each sum above into the indexes $i,j\in\mathcal{I}_{min}$ and $i,j\notin\mathcal{I}_{min},$ we obtain, for $\eta_*-\eta>0$ small,
\begin{equation*}
\begin{split}
&\left\|u_x(x,t)\right\|_p^p\sim C\mathcal{J}(\underline\alpha,t)^{-\frac{1}{\lambda}}\bigg|\mathcal{J}(\underline\alpha,t)^{-1}-C\left(\mathcal{J}(\underline\alpha,t)^{-1-\frac{1}{\lambda}}+C\right)\bigg|^p
\\
&\,\,\,\,\,\,\,\,+C\sum_{j\notin\mathcal{I}_{min}}\left\{(1-\lambda\eta h_j)^{-\frac{1}{\lambda}}\bigg|(1-\lambda\eta h_j)^{-1}-C\left(\mathcal{J}(\underline\alpha,t)^{-1-\frac{1}{\lambda}}+C\right)\bigg|^p\mu(\Omega_j)\right\}
\end{split}
\end{equation*}
where $\lambda\in(-\infty,0),$ $\mathcal{J}(\underline\alpha,t)=1-\lambda\eta(t)m_0$ and the constant $C>0$ may now also depend on $p\in[1,+\infty).$ Suppose $\lambda\in[-1,0),$ then $-1-\frac{1}{\lambda}\geq0$ and the above implies 
\begin{equation}
\label{eq:la1}
\left\|u_x(x,t)\right\|_p^p\sim C\mathcal{J}(\underline\alpha,t)^{-\left(p+\frac{1}{\lambda}\right)}+g(t)
\end{equation}
for $g(t)$ a bounded function on $[0,t_*)$ with finite, non-negative limit as $t\uparrow t_*.$ On the other hand, if $\lambda\in(-\infty,-1)$ then $-1-\frac{1}{\lambda}<0$ and 
\begin{equation}
\label{eq:la2}
\left\|u_x(x,t)\right\|_p^p\sim C\mathcal{J}(\underline\alpha,t)^{-\left(p+\frac{1}{\lambda}\right)}
\end{equation}
holds instead. The last part of the Theorem follows from (\ref{eq:la1}) and (\ref{eq:la2}) as $t\uparrow t_*.$ See \S \ref{subsec:pl} for examples.\qquad\end{proof}

\subsubsection{\textbf{Piecewise constant $u_0''(x)$}}\hfill
\label{subsubsec:pl}

When $u_0^{\prime\prime}(\alpha)\in PC_{\mathbb{R}}(0,1),$ the behaviour of solutions, in particular for $\lambda\in(1/2,1],$ can be rather different than the one described in theorems \ref{thm:theoremtwogen} and \ref{thm:pc}. Below, we consider a particular choice of data $u_0$ with a finite jump discontinuity in $u_0^{\prime\prime}$ at the point $\overline\alpha$ where $u_0^\prime$ attains its maximum $M_0.$ We find that the solution undergoes a two-sided blow-up in finite-time for $\lambda\in(1/2,+\infty).$ In particular, this signifies the formation of singularities in stagnation point-form solutions to the 2D incompressible Euler equations ($\lambda=1$) (\cite{Childress}, \cite{Weyl1}, \cite{Saxton1}). For $\lambda\in(-\infty,0),$ we find that a one-sided blow-up occurs at a finite number of locations in the domain.

Let
\begin{equation}
\label{eq:pieceu0}
u_0(\alpha)=
\begin{cases}
2\alpha^2-\alpha,\,\,\,\,\,\,&\alpha\in[0,1/2],
\\
-2\alpha^2+3\alpha-1,\,\,\,\,\,&\alpha\in(1/2,1]
\end{cases}
\end{equation}
so that $M_0=1$ and $m_0=-1$ occur at $\overline\alpha=1/2$ and $\underline\alpha=\{0,1\}$ respectively. Then $\mathcal{J}(\overline\alpha,t)=1-\lambda\eta(t),$\, $\mathcal{J}(\underline\alpha,t)=1+\lambda\eta(t)$ and $\eta_*=\frac{1}{\left|\lambda\right|}$ for $\lambda\neq0.$
%\begin{center}
%\begin{figure}[!ht]
%\includegraphics[scale=0.40]{piecewiseone.png} 
%\includegraphics[scale=0.40]{piecewisetwo.png} 
%\caption{$u_0(\alpha)$ and $u_0^\prime(\alpha)$ as given by (\ref{eq:pieceu0}).}
%\label{fig:piecewise}
%\end{figure}
%\end{center}
%\begin{figure}[ht]
%\centering
%\subfigure{
%\includegraphics[height=40mm,width=60mm]{piecewiseone.png}
%\label{fig:piecewiseone}
%}
%\subfigure{
%\includegraphics[height=40mm,width=60mm]{piecewisetwo.png}
%\label{fig:piecewisetwo}
%}
%\label{fig:piecewisedata}
%\caption{$u_0(\alpha)$ and $u_0^\prime(\alpha)$ as given by (\ref{eq:pieceu0}).}
%\end{figure}
Using (\ref{eq:pieceu0}), we find 
\begin{equation}
\label{eq:firstintp}
\bar{\mathcal{K}}_0(t)=
\begin{cases}
\frac{\mathcal{J}(\overline\alpha,t)^{1-\frac{1}{\lambda}}-\mathcal{J}(\underline\alpha,t)^{1-\frac{1}{\lambda}}}{2(1-\lambda)\eta(t)},\,\,\,\,\,\,\,&\lambda\in\mathbb{R}\backslash\{0,1\},
\\
\frac{1}{2\eta(t)}\ln\left(\frac{\eta_*+\eta(t)}{\eta_*-\eta(t)}\right),\,\,\,\,\,&\lambda=1
\end{cases}
\end{equation}
and
\begin{equation}
\label{eq:secondintp}
\begin{split}
\bar{\mathcal{K}}_1(t)=\frac{\mathcal{J}(\overline\alpha,t)^{-\frac{1}{\lambda}}-\mathcal{J}(\underline\alpha,t)^{-\frac{1}{\lambda}}}{2\eta(t)},\,\,\,\,\,\,\,\,\,\,\lambda\neq0.
\end{split}
\end{equation}
If $\lambda\in(-\infty,0),$ a one-sided blow-up, $m(t)\to-\infty,$ follows trivially from (\ref{eq:mainsolu}), (\ref{eq:firstintp})i) and (\ref{eq:secondintp}) as $t$ approaches a finite $t_*>0$ whose existence is guaranteed, in the limit as $\eta\uparrow\eta_*,$ by (\ref{eq:etaivp}) and (\ref{eq:firstintp})i). On the other hand, if $\lambda\in(0,+\infty)$ and $\eta_*-\eta>0$ is small,
\begin{equation}
\label{eq:firstintp2}
\bar{\mathcal{K}}_0(t)\sim
\begin{cases}
\frac{\lambda}{2(1-\lambda)}\mathcal{J}(\overline\alpha,t)^{1-\frac{1}{\lambda}},\,\,\,\,\,\,&\lambda\in(0,1),
\\
\frac{\lambda}{2^{\frac{1}{\lambda}}(\lambda-1)},\,\,\,\,\,&\lambda\in(1,+\infty)
\\
-C\log(\eta_*-\eta(t)),\,\,\,\,\,\,&\lambda=1
\end{cases}
\end{equation}
and
\begin{equation}
\label{eq:secondintp2}
\begin{split}
\bar{\mathcal{K}}_1(t)\sim\frac{\lambda}{2\mathcal{J}(\overline\alpha,t)^{\frac{1}{\lambda}}}.
\end{split}
\end{equation}
For $\alpha=\overline\alpha,$ the above estimates and (\ref{eq:mainsolu}) imply that, as $\eta\uparrow\eta_*,$
\begin{equation*}
M(t)=u_x(\gamma(\overline\alpha,t),t)\to
\begin{cases}
0,\,\,\,\,\,\,\,\,\,\,\,&\lambda\in(0,1/2),
\\
+\infty,\,\,\,\,\,\,\,\,\,\,\,&\lambda\in(1/2,+\infty).
\end{cases}
\end{equation*}
Furthermore, for $\alpha\neq\overline\alpha,$
\begin{equation*}
u_x(\gamma(\alpha,t),t)\to
\begin{cases}
0,\,\,\,\,\,\,\,\,\,\,\,&\lambda\in(0,1/2),
\\
-\infty,\,\,\,\,\,\,\,\,\,\,\,&\lambda\in(1/2,+\infty).
\end{cases}
\end{equation*}
For the threshold parameter $\lambda=1/2,$\, $u_x(\gamma,t)\to-1$ as $\eta\uparrow2$ for $\alpha\notin\{\overline\alpha,\underline\alpha\},$ whereas, $M(t)=u_x(\gamma(\overline\alpha,t),t)\equiv1$ and $m(t)=u_x(\gamma(\underline\alpha,t),t)\equiv-1.$ Finally, from (\ref{eq:etaivp}) and (\ref{eq:firstintp2})
\begin{equation*}
t_*-t\sim
\begin{cases}
C\int_{\eta}^{\eta_*}{(1-\lambda\mu)^{2(\lambda-1)}d\mu},\,\,\,\,\,\,\,\,\,\,&\lambda\in(0,1),
\\
C(\eta_*-\eta)(2-2\log(\eta_*-\eta)+\ln^2(\eta_*-\eta)),\,\,\,\,\,\,&\lambda=1,
\\
C(\eta_*-\eta),\,\,\,\,\,\,\,\,\,&\lambda\in(1,+\infty),
\end{cases}
\end{equation*}
and so $t_*=+\infty$ for $\lambda\in(0,1/2]$ but $0<t_*<+\infty$ when $\lambda\in(1/2,+\infty).$ 

\begin{remark}
By following an argument analogous to that of \S\ref{subsubsec:pc} for piecewise constant $u_0'$, it can be shown that the results from the above example extend to arbitrary data with piecewise constant $u_0''$. In fact, if instead of having $u_0'$ piecewise linear in $[0,1]$,  the above results extend to continuous $u_0'$ behaving linearly only in a small neighbourhood of $\overline\alpha_i$ for $\lambda>0$. Similarly for parameters $\lambda<0$ and $u_0'$ locally linear near $\underline\alpha_j$. Further details on the generalization of this result, as well as new results for data with arbitrary curvature near the locations in question, will be presented in a forthcoming paper.
\end{remark}
\begin{remark}
We recall that if $\lambda\in[1/2,1)$ and $u_0^{\prime\prime\prime}(x)\in L_{\mathbb{R}}^{\frac{1}{2(1-\lambda)}}(0,1),$ then $u$ persists globally in time (\cite{Okamoto2}). This result does not contradict the above blow-up example. Indeed, if $u_0^{\prime\prime\prime}\in L_{\mathbb{R}}^{\frac{1}{2(1-\lambda)}}$ for $\lambda\in[1/2,1),$ then $u_0''$ is an absolutely continuous function on $[0,1],$ and hence continuous. However, in the case just considered, $u_0''$ is, of course, not continuous.
\end{remark}
\begin{remark}
From Theorem \ref{thm:theoremtwogen}, which examines a family of smooth $u_0\in C^\infty_{\mathbb{R}},$ notice that $\lambda_*=1$ acts as the threshold parameter between solutions that vanish at $t=+\infty$ for $\lambda\in(0,\lambda_*)$ and those which blow-up in finite-time when $\lambda\in(\lambda_*,+\infty),$ while for $\lambda_*=1,\, u_x$ converges to a non-trivial steady-state as $t\to+\infty.$ In the example above with $u_0''\in PC_{\mathbb{R}},$ we have the corresponding behavior at $\lambda_*=1/2$ instead. Interestingly enough, when $\lambda=1/2$ or $\lambda=1,$ equation (\ref{eq:nonhomo}) i), iii) models stagnation point-form solutions to the 3D or 2D incompressible Euler equations respectively. An interesting question would be to examine the effect on blow-up of cusps in the graph of $u'_0$, for $\lambda=1/2.$ 
\end{remark}

\section{Examples}
\label{sec:examples}
Examples 1-4 in \S \ref{subsec:smooth} have $\lambda\in\{3,-5/2,1,-1/2\}$, respectively, and are instances of theorems \ref{thm:theoremtwogen}, \ref{thm:theoremfour} and \ref{thm:theoremthree}. In these cases, we will use formula (\ref{eq:finalsolu22}) and the \textsc{Mathematica} software to aid in the closed-form evaluation of some of the integrals and the generation of plots. Furthermore, examples 5 and 6 in \S \ref{subsec:pl} are representatives of Theorem \ref{thm:pc} for $\lambda=1$ and $-2.$ For simplicity, details of the computations in most examples are omitted. Finally, because solving the IVP (\ref{eq:etaivp}) is generally a difficult task, the plots in this section (with the exception of figure \ref{fig:blowfig2}$A$) will depict $u_x(\gamma(\alpha,t),t)$ for fixed $\alpha\in[0,1]$ against the variable $\eta(t)$, rather than $t$. Figure \ref{fig:blowfig2}$A$ will however illustrate $u(x,t)$  versus $x\in[0,1]$ for fixed $t\in[0,t_*)$.

\subsection{Examples for theorems \ref{thm:theoremtwogen}, \ref{thm:theoremthree} and \ref{thm:theoremfour}}
\label{subsec:smooth}
For examples 1-3, let $u_0(\alpha)=-\frac{1}{4\pi}\cos(4\pi\alpha).$ Then $u_0^\prime(\alpha)=\text{sin}(4\pi\alpha)$ attains its maximum $M_0=1$ at $\overline\alpha_i=\{1/8,5/8\},$ while $m_0=-1$ occurs at $\underline\alpha_j=\{3/8,7/8\}.$

\vspace{0.05in}

\textbf{Example 1.\, Two-sided Blow-up for $\lambda=3.$} Let $\lambda=3,$ then $\eta_*=\frac{1}{\lambda M_0}=1/3$ and for the time-dependent integrals we find that
\begin{equation}
\label{eq:good1}
\begin{split}
\bar{\mathcal{K}}_0(t)={}_2F_1\left[\frac{1}{6},\frac{2}{3};1;9\,\eta(t)^2\right]\to\frac{\Gamma\left(\frac{1}{6}\right)}{\Gamma\left(\frac{1}{3}\right)\,\Gamma\left(\frac{5}{6}\right)}\sim1.84
\end{split}
\end{equation}
and
\begin{equation}
\label{eq:good2}
\begin{split}
\int_0^1{\frac{u_0^\prime\,d\alpha}{(1-3\eta(t)u_0^\prime)^{\frac{4}{3}}}}=2\eta(t)\,{}_2F_1\left[\frac{7}{6},\frac{5}{3};2;9\,\eta(t)^2\right]\to+\infty
\end{split}
\end{equation}
as $\eta\uparrow1/3.$ Using (\ref{eq:good1}) and (\ref{eq:good2}) on (\ref{eq:finalsolu22}), and taking the limit as $\eta\uparrow1/3,$ we see that $M(t)=u_x(\gamma(\overline\alpha_i,t),t)\to+\infty$ whereas, for $\alpha\neq\overline\alpha_i,$ $u_x(\gamma(\alpha,t),t)\to-\infty.$ The blow-up time $t_*\sim0.54$ is obtained from (\ref{eq:etaivp}) and (\ref{eq:good1}). See figure \ref{fig:blowfig}$A$.

\vspace{0.05in}

\textbf{Example 2.\, Two-sided Blow-up for $\lambda=-5/2.$} For $\lambda=-5/2,$ $\eta_*=\frac{1}{\lambda m_0}=2/5$. Then, we now have that
\begin{equation}
\label{eq:good3}
\begin{split}
\bar{\mathcal{K}}_0(t)={}_2F_1\left[-\frac{1}{5},\frac{3}{10};1;\frac{25}{4}\eta(t)^2\right]\to\frac{\Gamma\left(\frac{9}{10}\right)}{\Gamma\left(\frac{7}{10}\right)\,\Gamma\left(\frac{6}{5}\right)}\sim0.9
\end{split}
\end{equation}
and
\begin{equation}
\label{eq:good4}
\begin{split}
\int_0^1{\frac{u_0^\prime\,d\alpha}{(1+5\eta(t)u_0^\prime/2)^{\frac{3}{5}}}}=-\frac{3}{4}\eta(t)\,{}_2F_1\left[\frac{4}{5},\frac{13}{10};2;\frac{25}{4}\eta(t)^2\right]\to-\infty
\end{split}
\end{equation}
as $\eta\uparrow2/5.$ Plugging the above formulas into (\ref{eq:finalsolu22}) and letting $\eta\uparrow2/5,$ we find that $m(t)=u_x(\gamma(\underline\alpha_j,t),t)\to-\infty$ while, for $\alpha\neq\underline\alpha_j,$ $u_x(\gamma(\alpha,t),t)\to+\infty.$ The blow-up time $t_*\sim0.46$ is obtained from (\ref{eq:etaivp}) and (\ref{eq:good3}). See figure \ref{fig:blowfig}$B$. 

\vspace{0.05in}

\textbf{Example 3.\, Global Existence for $\lambda=1.$} Let $\lambda=1,$ then 
\begin{equation}
\label{eq:ge}
\begin{split}
\bar{\mathcal{K}}_0(t)=\frac{1}{\sqrt{1-\eta(t)^2}}\,\,\,\,\,\,\,\,\text{and}\,\,\,\,\,\,\,\,\int_0^1{\frac{u_0^\prime\,d\alpha}{(1-\eta(t)u_0^\prime)^2}}=\frac{\eta(t)}{(1-\eta(t)^2)^{\frac{3}{2}}}
\end{split}
\end{equation}
both diverge to $+\infty$ as $\eta\uparrow\eta_*=1.$ Also, (\ref{eq:ge})i) and (\ref{eq:etaivp}) imply $\eta(t)=\text{tanh}\,t,$ which we use on (\ref{eq:finalsolu22}), along with (\ref{eq:ge}), to obtain $$u_x(\gamma(\alpha,t),t)=\frac{\text{tanh}\,t-\text{sin}(4\pi\alpha)}{\text{tanh}\,t\,\text{sin}(4\pi\alpha)-1}.$$ Then, $M(t)=u_x(\gamma(\overline\alpha_i,t),t)\equiv1$ and $m(t)=u_x(\gamma(\underline\alpha_j,t),t)\equiv-1$ while, for $\alpha\notin\{\overline\alpha_i,\underline\alpha_j\}$, $u_x(\gamma(\alpha,t),t)\to-1$ as $\eta\uparrow 1$. Finally, $\eta(t)=\text{tanh}\,t$ yields $t_*=\lim_{\eta\uparrow 1}{\text{arctanh}\,\eta}=+\infty.$ It is also easy to see from the formulas in \S \ref{sec:sol} and (\ref{eq:ge})i) that $I(t)\equiv-1$ for $I(t)$ the nonlocal term (\ref{eq:nonhomo})iii). See figure \ref{fig:blowfig}$C$.

\vspace{0.05in}

\textbf{Example 4.\, One-sided Blow-up for $\lambda=-1/2.$} For $\lambda=-1/2$ (HS equation), let $u_0=\cos(2\pi\alpha)+2\cos(4\pi\alpha).$ Then, the least value $m_0<0$ of $u_0'$ and the location $\underline\alpha\in[0,1]$ where it occurs are given, approximately, by $m_0\sim-30$ and $\underline\alpha\sim0.13,$ respectively. Also, $\eta_*=-\frac{2}{m_0}\sim0.067$ and 
\begin{equation}
\label{eq:ba0}
\begin{split}
\bar{\mathcal{K}}_0(t)=1+\frac{17\pi^2\eta(t)^2}{2},\,\,\,\,\,\,\,\,\,\,\,\,\,\int_0^1{u_0^{\prime}(\alpha)\left(1+\eta(t)\frac{u_0^{\prime}(\alpha)}{2}\right)\,d\alpha}=17\pi^2\eta(t).
\end{split}
\end{equation}	
Then, (\ref{eq:etaivp}) and (\ref{eq:ba0})i) give $\eta(t)=\sqrt{\frac{2}{17\pi^2}}\,\tan\left(\pi\sqrt{\frac{17}{2}}\,t\right)$. Using these results on (\ref{eq:finalsolu22}) yields, after simplification,
\begin{equation*}
\begin{split}
u_x(\gamma(\alpha,t),t)=\frac{\pi\left(2\sin(2\pi\alpha)+8\sin(4\pi\alpha)+\sqrt{34}\tan\left(\pi\sqrt{\frac{17}{2}}t\right)\right)}{\sqrt{\frac{2}{17}}\,\tan\left(\pi\sqrt{\frac{17}{2}}\,t\right)\left(\sin(2\pi\alpha)+4\sin(4\pi\alpha)\right)-1}
\end{split}
\end{equation*}	
for $0\leq\eta<\eta_*.$ We find that\, $m(t)=u_x(\gamma(\underline\alpha,t),t)\to-\infty$\, as $\eta\uparrow\eta_*,$ whereas, for $\alpha\neq\underline\alpha,$\, $u_x(\gamma(\alpha,t),t)$ remains finite. Finally, $t_*=t\left(-2/m_0\right)\sim0.06.$ See figure \ref{fig:blowfig}$D$. 

\begin{center}
\begin{figure}[!ht]
\includegraphics[scale=0.4]{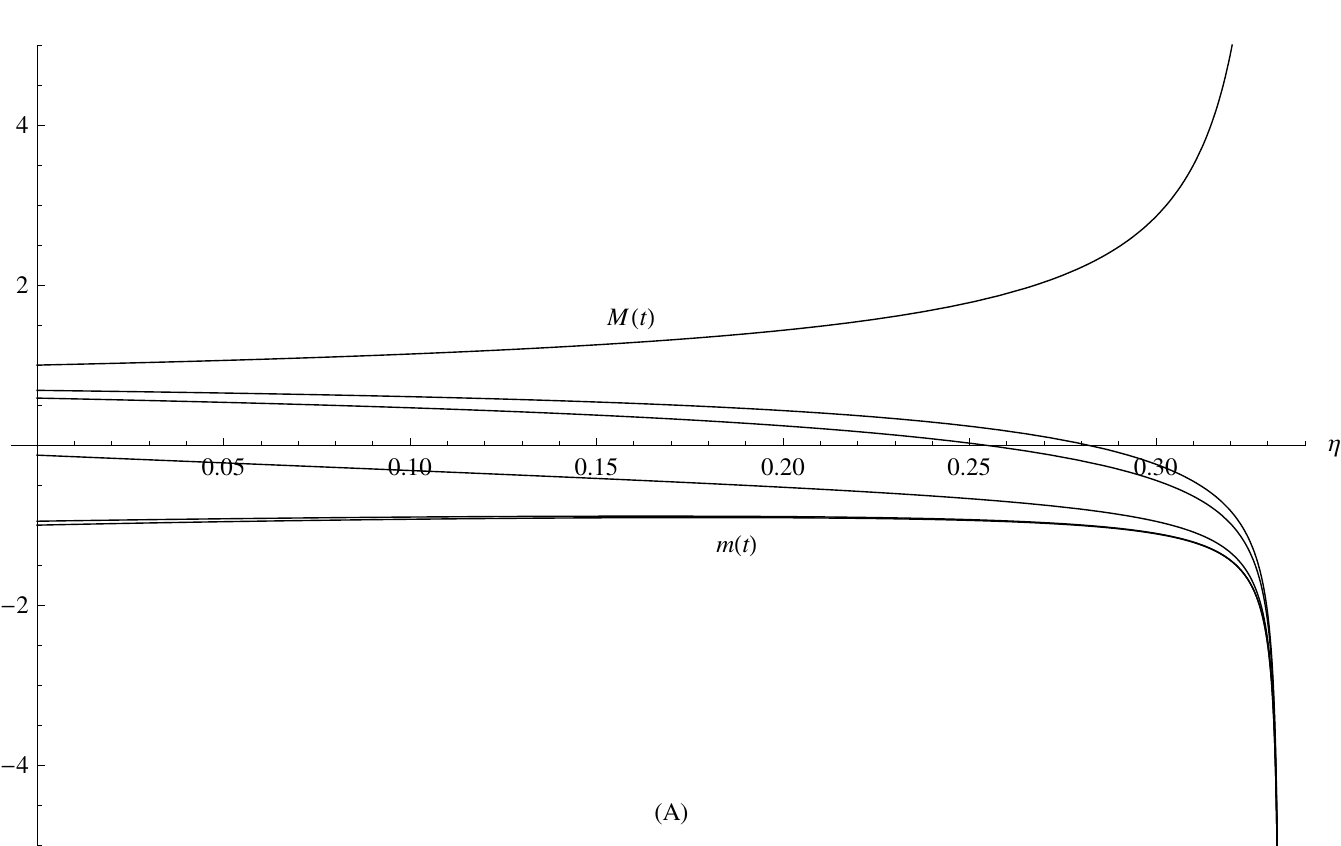} 
\includegraphics[scale=0.4]{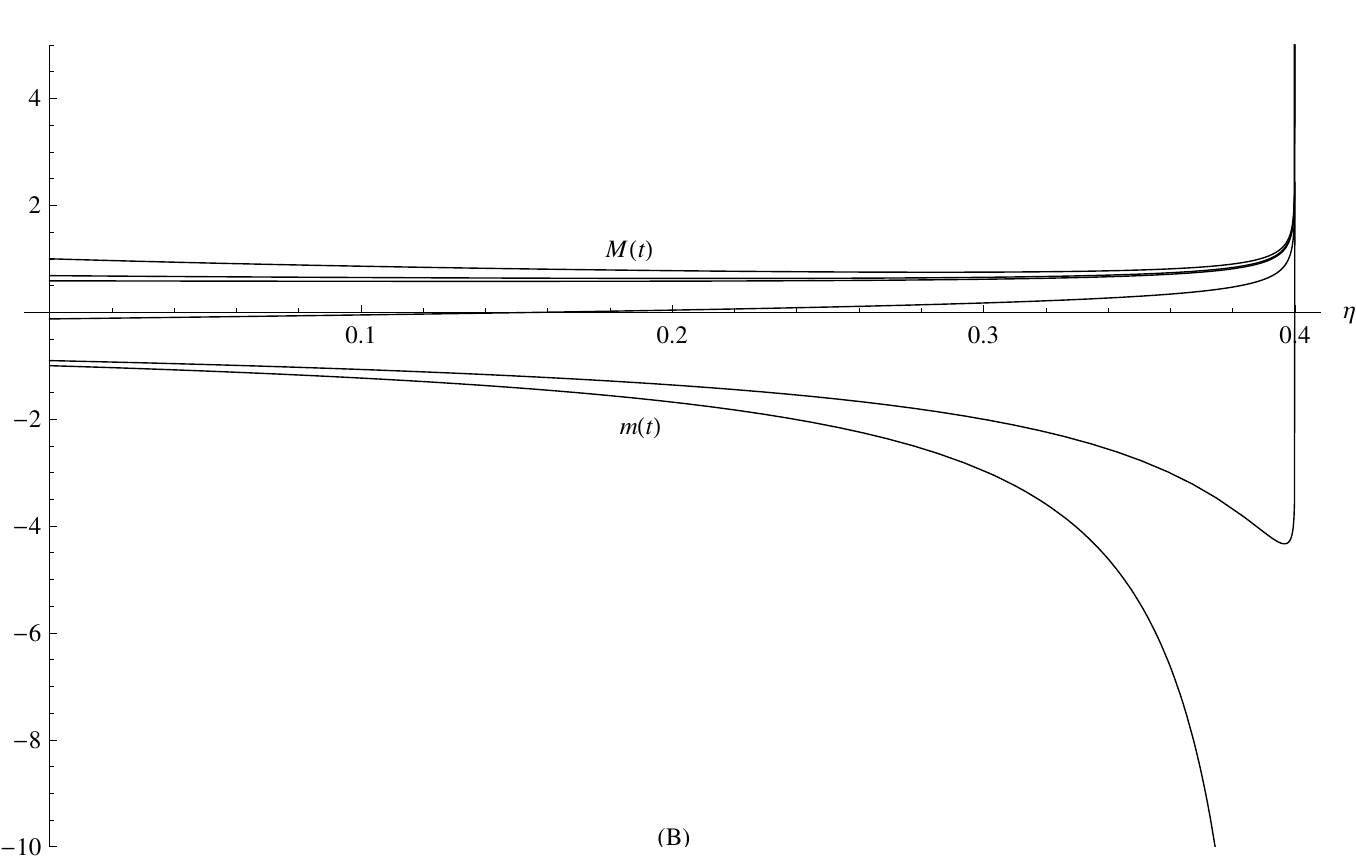} \\ 
\includegraphics[scale=0.4]{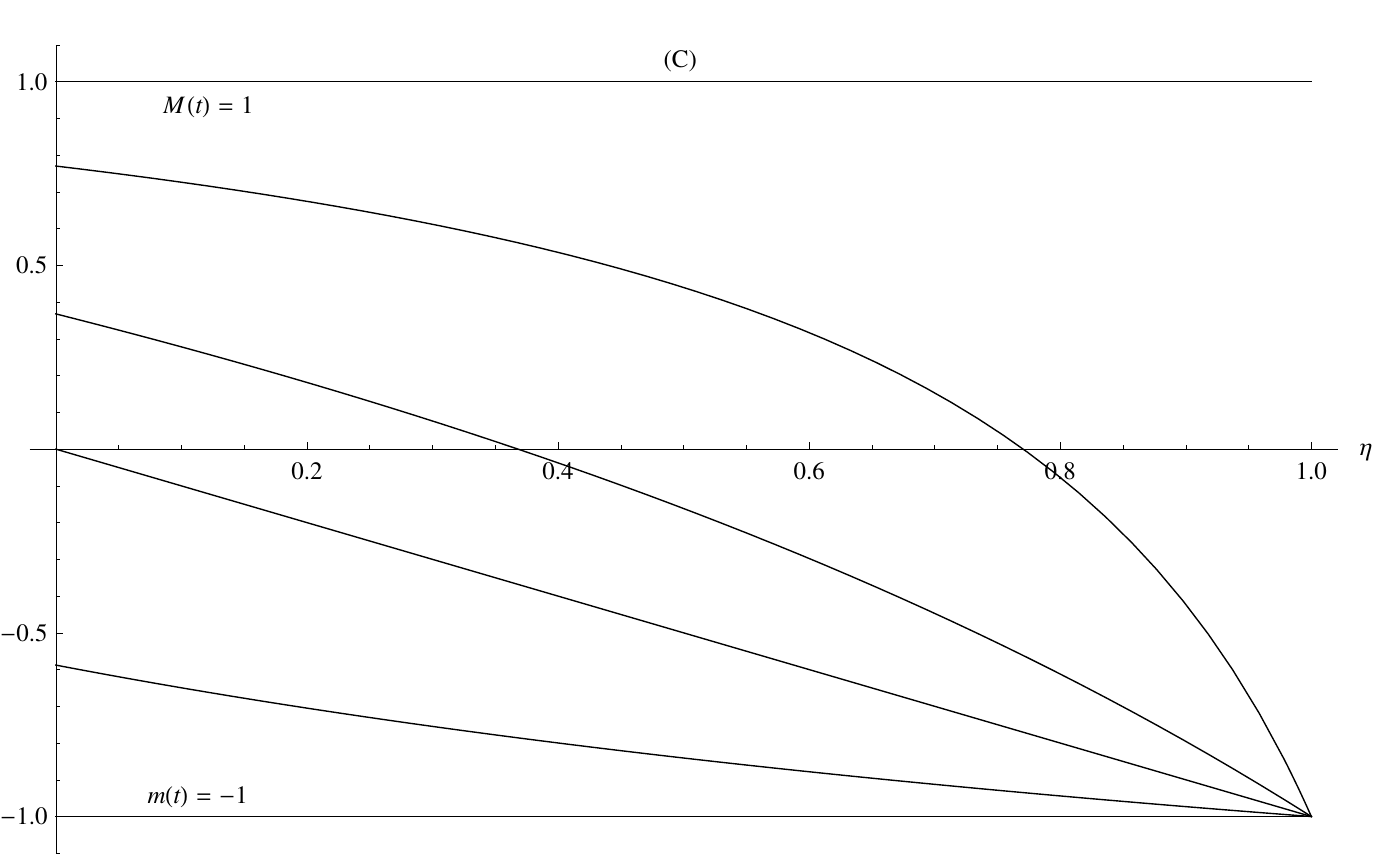}
\includegraphics[scale=0.4]{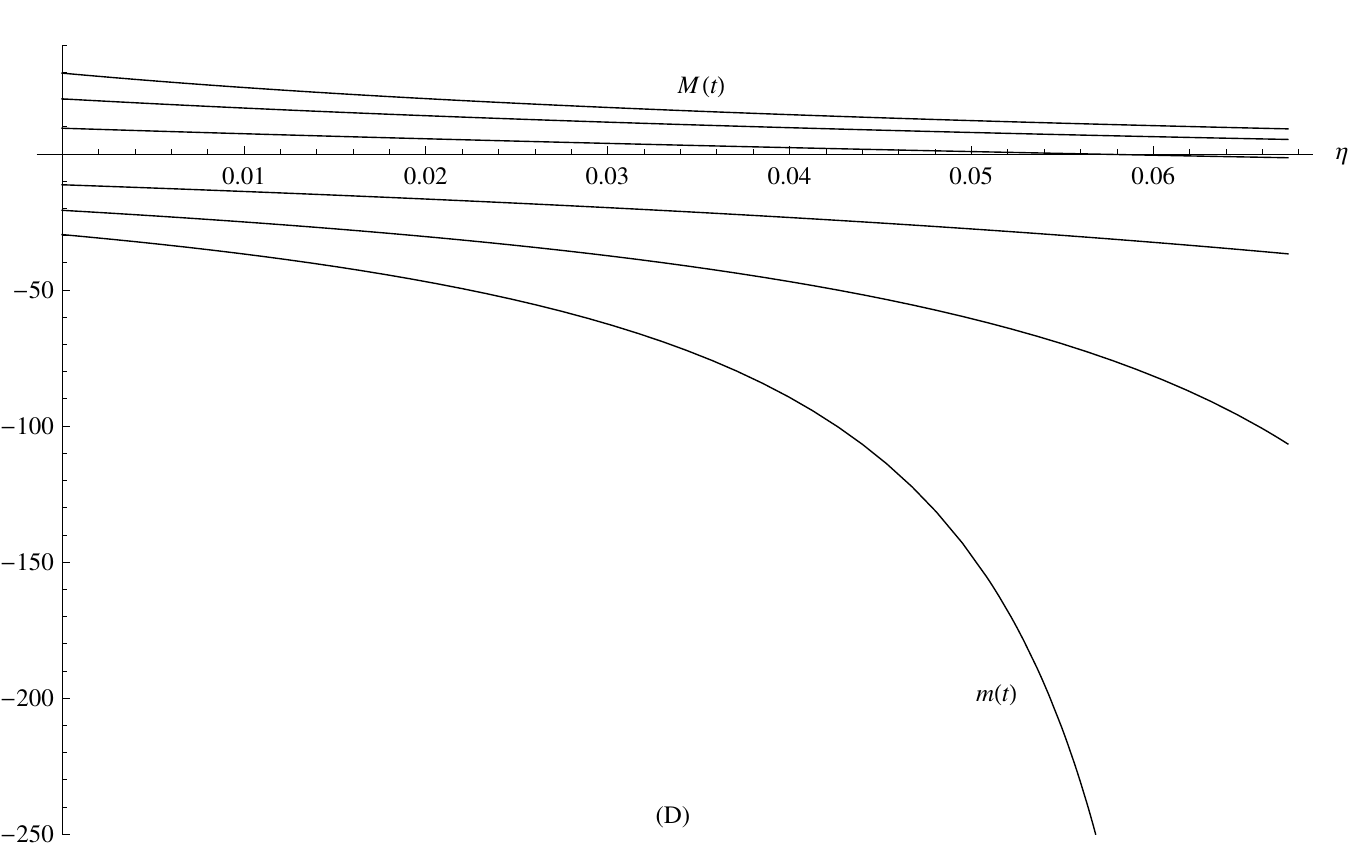} 
\caption{Figures $A$ and $B$ depict two-sided, everywhere blow-up of (\ref{eq:finalsolu22}) for $\lambda=3$ and $-5/2$ (Examples 1 and 2) as $\eta\uparrow 1/3$ and $2/5$, respectively. Figure $C$ (Example 3) represents global existence in time for $\lambda=1$, whereas figure $D$ (Example 4) illustrates one-sided, discrete blow-up for $\lambda=-1/2$ as $\eta\uparrow0.067$.}
\label{fig:blowfig}
\end{figure}
\end{center}

\subsection{Examples for Theorem \ref{thm:pc}}
\label{subsec:pl}
For examples 5 and 6 below, we let  
\begin{equation}
\label{eq:dataplex}
u_0(\alpha)=
\begin{cases}
-\alpha,\,\,\,\,\,\,\,\,\,\,&\,\,\,\,\,\,0\leq\alpha<1/4,
\\
\alpha-1/2,\,\,\,\,\,\,\,\,\,\,\,\,&1/4\leq\alpha<3/4.
\\
1-\alpha,\,\,\,\,\,\,\,\,\,&3/4\leq\alpha\leq1.
\end{cases}
\end{equation}
Then, $M_0=1$ occurs when $\alpha\in[1/4,3/4)$, $m_0=-1$ for $\alpha\in[0,1/4)\cup[3/4,1]$ and $\eta_*=\frac{1}{\left|\lambda\right|}$ for $\lambda\neq0$. Also, notice that (\ref{eq:dataplex}) is odd about the midpoint $\alpha=1/2$ and it vanishes at the end-points (as it should due to periodicity). As a result, uniqueness of solution to (\ref{eq:cha}) implies that $\gamma(0,t)\equiv0$ and $\gamma(1,t)\equiv1$ for as long as $u$ is defined. 
\vspace{0.05in}

\textbf{Example 5.\,Global estimates for $\lambda=1.$}
Using (\ref{eq:dataplex}), we find that $\bar{\mathcal{K}}_0(t)=(1-\eta(t)^2)^{-1}$ for $0\leq\eta<\eta_*=1.$ Then (\ref{eq:sum}) implies $\gamma_{\alpha}(\alpha,t)=\frac{1-\eta(t)^2}{1-\eta(t)u_0'(\alpha)}$, or 
\begin{equation}
\label{eq:chaex1pc}
\gamma(\alpha,t)=
\begin{cases}
(1-\eta(t))\alpha,\,\,\,\,\,\,\,\,\,&\,\,\,\,\,\,0\leq\alpha<1/4,
\\
\alpha+\eta(t)(\alpha-1/2),\,\,\,\,\,\,\,\,\,\,&1/4\leq\alpha<3/4,
\\
\alpha+\eta(t)(1-\alpha),\,\,\,\,\,\,\,\,&3/4\leq\alpha\leq1
\end{cases}
\end{equation}
after integrating and using (\ref{eq:dataplex}) and $\gamma(0,t)\equiv0$. Since $\dot\gamma=u\circ\gamma,$ we have that
\begin{equation}
\label{eq:uex1pc}
u(\gamma(\alpha,t),t)=
\begin{cases}
-\alpha\dot\eta(t),\,\,\,\,\,\,\,\,\,&\,\,\,\,\,\,0\leq\alpha<1/4,
\\
(\alpha-1/2)\dot\eta(t),\,\,\,\,\,\,\,\,\,\,\,&1/4\leq\alpha<3/4
\\
(1-\alpha)\dot\eta(t),\,\,\,\,\,\,\,\,&3/4\leq\alpha\leq1
\end{cases}
\end{equation}
where, by (\ref{eq:etaivp}) and $\bar{\mathcal{K}}_0$ above, $\dot\eta(t)=(1-\eta(t)^2)^2$. Now, (\ref{eq:chaex1pc}) lets us solve for $\alpha=\alpha(x,t)$, the inverse Lagrangian map, as
\begin{equation}
\label{eq:invex1pc}
\alpha(x,t)=
\begin{cases}
\frac{x}{1-\eta(t)},\,\,\,\,\,\,\,\,\,&\,\,\,\,\,\,\,\,\,\,\,\,\,0\leq x<\frac{1-\eta(t)}{4},
\\
\frac{2x+\eta(t)}{2(1+\eta(t))},\,\,\,\,\,\,\,\,\,&\frac{1-\eta(t)}{4}\leq x<\frac{3+\eta(t)}{4},
\\
\frac{x-\eta(t)}{1-\eta(t)},\,\,\,\,\,\,\,\,\,\,&\frac{3+\eta(t)}{4}\leq x\leq1,
\end{cases}
\end{equation}
which we use on (\ref{eq:uex1pc}) to obtain the corresponding Eulerian representation
\begin{equation}
\label{eq:uex1pceul}
u(x,t)=
\begin{cases}
-(1-\eta(t))(1+\eta(t))^2x,\,\,\,\,\,\,\,\,\,&\,\,\,\,\,\,\,\,\,\,\,\,\,0\leq x<\frac{1-\eta(t)}{4},
\\
\frac{1}{2}(1+\eta(t))(1-\eta(t))^2(2x-1),\,\,\,\,\,\,\,\,\,&\frac{1-\eta(t)}{4}\leq x<\frac{3+\eta(t)}{4},
\\
(1-\eta(t))(1+\eta(t))^2(1-x),\,\,\,\,\,\,\,\,\,\,\,&\frac{3+\eta(t)}{4}\leq x\leq1,
\end{cases}
\end{equation}
which in turn yields
\begin{equation}
\label{eq:uxex1pceul}
u_x(x,t)=
\begin{cases}
-(1-\eta(t))(1+\eta(t))^2,\,\,\,\,\,\,\,\,\,&\,\,\,\,\,\,\,\,\,\,\,\,\,0\leq x<\frac{1-\eta(t)}{4},
\\
(1+\eta(t))(1-\eta(t))^2,\,\,\,\,\,\,\,\,\,&\frac{1-\eta(t)}{4}\leq x<\frac{3+\eta(t)}{4},
\\
-(1-\eta(t))(1+\eta(t))^2,\,\,\,\,\,\,\,\,\,\,\,&\frac{3+\eta(t)}{4}\leq x\leq1.
\end{cases}
\end{equation}
Finally, solving the IVP for $\eta$ we obtain $t(\eta)=\frac{1}{2}\left(\text{arctanh}(\eta)+\frac{\eta}{1-\eta^2}\right),$ so that (\ref{eq:assympt}) gives $t_*=\lim_{\eta\uparrow 1}t(\eta)=+\infty$. See figure \ref{fig:blowfig2}$A$ below.

\begin{remark}
The vanishing of the characteristics in example 5 greatly facilitates the computation of an explicit solution formula for $\dot\gamma(\alpha,t)=u(\gamma(\alpha,t),t)$. However, as it is generally the case, $\gamma(0,t)$ may not be identically zero. In that case, integration of (\ref{eq:sum}) now yields
\begin{equation}
\label{eq:meander}
\gamma(\alpha,t)=\gamma(0,t)+\frac{1}{\bar{\mathcal{K}}_0(t)}\int_0^{\alpha}{\frac{dy}{\mathcal{J}(y,t)^{\frac{1}{\lambda}}}},
\end{equation}
which we differentiate in time to obtain
\begin{equation}
\label{eq:meander3}
u(\gamma(\alpha,t),t)=\dot\gamma(0,t)+\frac{d}{dt}\left(\frac{1}{\bar{\mathcal{K}}_0(t)}\int_0^{\alpha}{\mathcal{K}_0(y,t)dy}\right).
\end{equation}
In order to determine the time-dependent function $\dot\gamma(0,t)$, we may use, for instance, the `conservation in mean' condition\footnote[13]{In \cite{Saxton1}, the authors showed that (\ref{eq:meander2}) follows naturally from the study of spatially periodic, stagnation point-form solutions to the $n$ dimensional Euler equations with spatially periodic pressure term.}
\begin{equation}
\label{eq:meander2}
\int_0^1{u_0(x)dx}=\int_0^1{u(x,t)dx}=\int_0^1{u(\gamma(\alpha,t),t)\gamma_\alpha(\alpha,t) d\alpha}.
\end{equation}
Let us assume (\ref{eq:meander2}) holds. Then, multiplying (\ref{eq:meander3}) by the mean-one function $\gamma_{\alpha}$ in (\ref{eq:sum}), integrating in $\alpha$ and using (\ref{eq:meander2}), yields
\begin{equation}
\label{eq:meander4}
\dot\gamma(0,t)=\int_0^1{u_0(\alpha)d\alpha}-\int_0^1{\frac{\mathcal{K}_0(\alpha,t)}{\bar{\mathcal{K}}_0(t)}\frac{d}{dt}\left(\frac{1}{\bar{\mathcal{K}}_0(t)}\int_0^{\alpha}{\mathcal{K}_0(y,t)dy}\right)d\alpha}.
\end{equation}
Omitting the details of the computations, (\ref{eq:meander4}) may, in turn, be written as
\begin{equation}
\label{eq:meander6}
\dot\gamma(0,t)=\int_0^1{u_0(\alpha)d\alpha}+\frac{\bar{\mathcal{K}}_0(t)^{^{-2(1+\lambda)}}}{\lambda\eta(t)}\left(\frac{\bar{\mathcal{K}}_0(t)\bar{\mathcal{K}}_1(t)}{2}-\int_0^1{\mathcal{K}_0(\alpha,t)\int_0^{\alpha}{\mathcal{K}_1(y,t)\,dy}\,\,d\alpha}\right)
\end{equation}
The above and (\ref{eq:meander3}) yield a representation formula for $u(\gamma,t)$.  Integrating (\ref{eq:meander6}) in time and using (\ref{eq:meander}) gives an expression for the characteristics $\gamma$. Finally, we remark that under Dirichlet boundary conditions and/or using initial data $u_0$ which is odd about the midpoint (\cite{Aconstantin1}, \cite{Wunsch1}), a general formula for $u(\gamma,t)$ can be obtained from (\ref{eq:meander3}) by simply setting $\dot\gamma(0,t)\equiv0$.
\end{remark}

\vspace{0.05in}

\textbf{Example 6.\,Finite-time blow-up for $\lambda=-2.$} Using (\ref{eq:dataplex}) and $\lambda=-2$,   
\begin{equation*}
\begin{split}
\bar{\mathcal{K}}_0(t)=\frac{\sqrt{1-2\eta(t)}+\sqrt{1+2\eta(t)}}{2}\,\,\,\,\,\,\,\,\text{and}\,\,\,\,\,\,\,\,\int_0^1{\frac{u_0'(\alpha)\,d\alpha}{\mathcal{J}(\alpha,t)^{1+\frac{1}{\lambda}}}}=\frac{\,d\bar{\mathcal{K}}_0(t)}{d\eta}
\end{split}
\end{equation*}
for $\eta\in[0,\eta_*)$ and $\eta_*=1/2$. Then, (\ref{eq:finalsolu22}) yields 
\begin{equation}
\label{eq:ux6}
u_x(\gamma(\alpha,t),t)=
\begin{cases}
M(t)=\frac{\left(\sqrt{1-2\eta(t)}+\sqrt{1+2\eta(t)}\right)^3}{8(1+2\eta(t))\sqrt{1-2\eta(t)}},\,\,\,\,\,\,\,\,\,&\alpha\in[1/4,3/4),
\\
m(t)=-\frac{\left(\sqrt{1-2\eta(t)}+\sqrt{1+2\eta(t)}\right)^3}{8(1-2\eta(t))\sqrt{1+2\eta(t)}},\,\,\,\,\,\,&\alpha\in[0,1/4)\cup[3/4,1],
\end{cases}
\end{equation}
so that $M(t)\to+\infty$ and $m(t)\to-\infty$ as $\eta\uparrow1/2$. The finite blow-up time $t_*>0$ is obtained from (\ref{eq:etaivp}) and $\bar{\mathcal{K}}_0$ above. We find
$$t(\eta)=\frac{1}{6\eta^3}\left(\eta^2\left(6-4\sqrt{1-4\eta^2}\right)+\sqrt{1-4\eta^2}-1\right),$$
so that $t_*=t(1/2)=2/3$. See figure \ref{fig:blowfig2}$B$ below.
\begin{center}
\begin{figure}[!ht]
\includegraphics[scale=0.27]{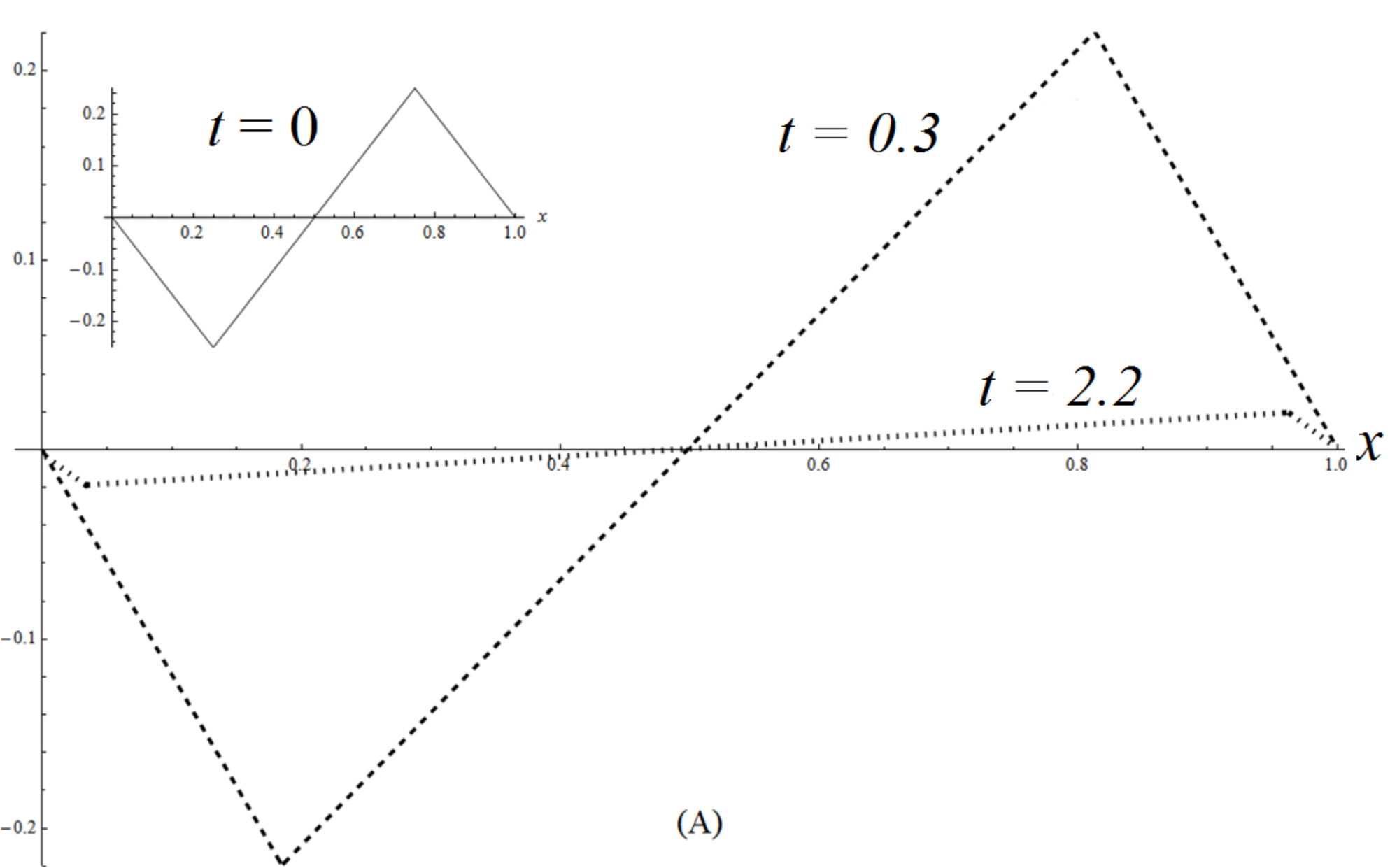}
\includegraphics[scale=0.32]{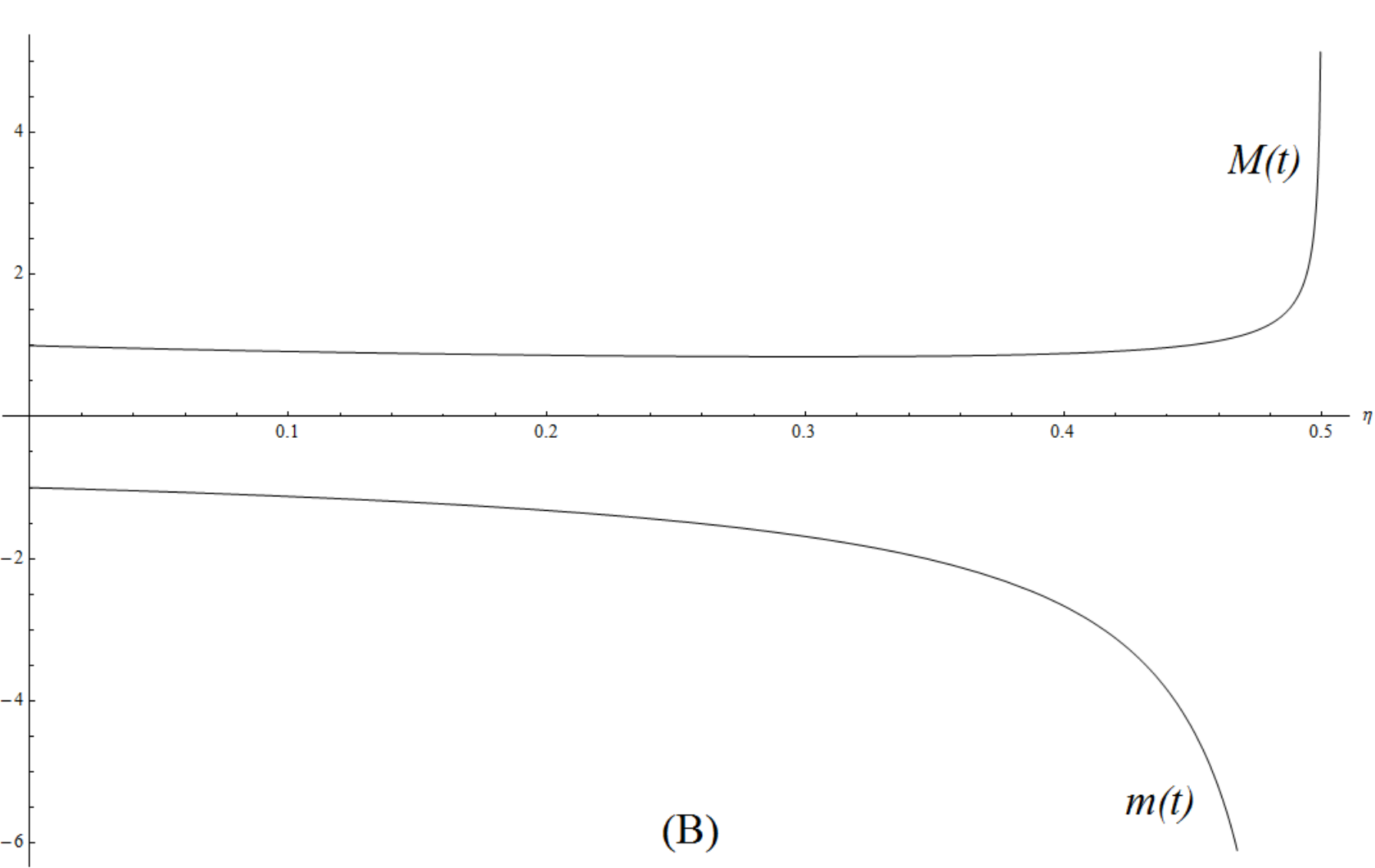} 
\caption{In figure $A$, (\ref{eq:uex1pceul}) vanishes as $t\to+\infty$, while figure $B$ depicts two-sided blow-up of (\ref{eq:ux6}) as $\eta\uparrow\eta_*=1/2.$}
\label{fig:blowfig2}
\end{figure}
\end{center}

\appendix\section{\textit{Global existence for $\lambda=0$ and smooth data.}}  
Setting $\lambda=0$ in (\ref{eq:chain}) gives 
$$(\ln\gamma_\alpha)^{\ddot{}}=-\int_0^1{u_x^2dx}=I(t).$$ 
Then, integrating twice in time and using $\dot\gamma_\alpha=(u_x(\gamma,t))\gamma_\alpha$ and $\gamma_\alpha(\alpha,0)=1,$ yields
\begin{equation}
\label{eq:02}
\begin{split}
\gamma_\alpha(\alpha,t)=e^{tu_0^\prime(\alpha)}e^{\int_0^t{(t-s)I(s)ds}}.
\end{split}
\end{equation}
Since $\gamma_\alpha$ has mean one in $[0,1],$ integrate (\ref{eq:02}) in $\alpha$ to obtain
\begin{equation}
\label{eq:03}
\begin{split}
e^{\int_0^t{(t-s)I(s)ds}}=\left(\int_0^1{e^{tu_0^\prime(\alpha)}d\alpha}\right)^{-1}.
\end{split}
\end{equation}
Combining this with (\ref{eq:02}) gives $\gamma_{\alpha}(\alpha,t)=e^{tu_0^\prime(\alpha)}\left(\int_0^1{e^{tu_0^\prime(\alpha)}d\alpha}\right)^{-1},$ a bounded expression for $\gamma_\alpha$ which we differentiate w.r.t. to $t,$ to get
\begin{equation}
\label{eq:05}
\begin{split}
u_x(\gamma(\alpha,t),t)=u_0^\prime(\alpha)-\frac{\int_0^1{u_0^\prime(\alpha)e^{tu_0^\prime(\alpha)}d\alpha}}{\int_0^1{e^{tu_0^\prime(\alpha)}d\alpha}}
\end{split}
\end{equation}
and so 
$$0\leq u_0^\prime(\alpha)-u_x(\gamma(\alpha,t),t)\leq\int_0^1{u_0^\prime(\alpha)\,e^{tu_0^\prime(\alpha)}d\alpha},\,\,\,\,\,\,\,\,\,\,\,t\geq0.$$ 
The cases where $u'_0, u''_0\in PC_\mathbb{R}$  are analogous, and follow from the above.

\section{\textit{Proof of Lemma \ref{prop:prop}.}} 

For the series (\ref{eq:2f1}), we have the following convergence results \cite{Magnus1}:
				\begin{itemize}
				\item Absolute convergence for all $\lvert z\rvert<1.$
				\item Absolute convergence for $\lvert z\rvert=1$ if $Re(a+b-c)<0.$
				\item Conditional convergence for $\lvert z\rvert=1,\, z\neq1$ if $0\leq Re(a+b-c)<1.$
				\item Divergence if $\lvert z\rvert=1$ and $1\leq Re(a+b-c).$
				\end{itemize}
Furthermore, consider the identities \cite{Magnus1}:
\begin{equation}
\label{eq:2f112}
\begin{split}
\frac{d}{dz}{}_2F_1\left[a,b;c;z\right]=\frac{ab}{c}\,{}_2F_1\left[a+1,b+1;c+1;z\right],\,\,\,\,{}_2F_1\left[a,b;b;z\right]=(1-z)^{-a},
\end{split}
\end{equation} 
as well as the contiguous relations 
\begin{equation}
\label{eq:2f122}
\begin{split}
z\,{}_2F_1\left[a+1,b+1;c+1;z\right]=\frac{c}{a-b}\left({}_2F_1\left[a,b+1;c;z\right]-{}_2F_1\left[a+1,b;c;z\right]\right)
\end{split}
\end{equation} 
and
\begin{equation}
\label{eq:2f1222}
\begin{split}
{}_2F_1\left[a,b;c;z\right]=\frac{b}{b-a}\,{}_2F_1\left[a,b+1;c;z\right]-\frac{a}{b-a}\,{}_2F_1\left[a+1,b;c;z\right]
\end{split}
\end{equation} 
for $b\neq a.$ For simplicity, we prove the Lemma for $\beta_0=0$ and write $F={}_2F_1$. 

Set $a=1/2$, $c=3/2$ and $z=-\frac{C_0\beta^2}{\epsilon}$. Then, the assumptions in the Lemma imply that $-1\leq z\leq0$ and $a+b-c=b-1<1$, so that (\ref{eq:2f112})i) and the chain rule give
\begin{equation}
\label{eq:2f13}
\begin{split}
\frac{d}{d\beta}\,&\left\{\beta F\left[a,b;c;z\right]\right\}=\frac{2b}{3}\left(zF\left[a+1,b+1;c+1;z\right]\right)+F\left[a,b;c;z\right].
\end{split}
\end{equation} 
But for $b\neq a=\frac{1}{2},$ (\ref{eq:2f122}) and (\ref{eq:2f1222}) imply 
\begin{equation*}
\begin{split}
zF\left[a+1,b+1;c+1;z\right]=\frac{3}{1-2b}\left(F\left[a,b+1;c;z\right]-F\left[b,c;c;z\right]\right)
\end{split}
\end{equation*} 
and
\begin{equation*}
\begin{split}
F\left[a,b;c;z\right]=\frac{2b}{2b-1}F\left[a,b+1;c;z\right]-\frac{1}{2b-1}F\left[b,c;c;z\right].
\end{split}
\end{equation*} 
Substituting the two above into (\ref{eq:2f13}) and using (\ref{eq:2f112})ii)\footnote[14]{Notice that no issue arises when using identity (\ref{eq:2f112}) because, in our case, $-1\leq z\leq0$.} yields our result.\,\,\,\,\,\,\,\,$\square$

\section{\textit{Proof of (\ref{eq:maxp}).}}

We prove (\ref{eq:maxp}) for $\lambda>0$. The case of parameter values $\lambda<0$ follows similarly.

Suppose $\lambda>0$ and set $\eta_{\epsilon}=\frac{1}{\lambda M_0+\epsilon}$ for arbitrary $\epsilon>0.$ Then $0<\eta_{\epsilon}<\eta_*$ for $\eta_*=\frac{1}{\lambda M_0}$. Also, due to the definition of $M_0$, $$1-\lambda\eta_{\epsilon}u_0^\prime(\alpha)=\frac{\epsilon+\lambda(M_0-u_0'(\alpha))}{\lambda M_0+\epsilon}>0$$
for all $\alpha\in[0,1]$, while $1-\lambda\eta_{\epsilon}u_0^\prime(\alpha)=0$ only if $\epsilon=0$ and $\alpha=\overline\alpha_i.$ We conclude that
\begin{equation}
\label{eq:barg}
\begin{split}
1-\lambda\eta(t)u_0^\prime(\alpha)>0
\end{split}
\end{equation}
for all $0\leq\eta(t)<\eta_*$ and $\alpha\in[0,1].$ But $u_0^\prime(\alpha)\leq M_0$, or equivalently 
\begin{equation*}
\begin{split}
u_0^\prime(\alpha)(1-\lambda\eta(t)M_0 )\leq M_0(1-\lambda\eta(t) u_0^\prime(\alpha)),
\end{split}
\end{equation*}
therefore (\ref{eq:barg}) and $u_0'(\overline\alpha_i)=M_0$, $1\leq i\leq m$, yield
\begin{equation}
\label{eq:bef22}
\begin{split}
\frac{u_0^\prime(\alpha)}{\mathcal{J}(\alpha,t)}\leq\frac{u_0'(\overline\alpha_i)}{\mathcal{J}(\overline\alpha_i,t)}
\end{split}
\end{equation}
for $0\leq\eta<\eta_*$ and $\mathcal{J}(\alpha,t)=1-\lambda\eta(t)u_0'(\alpha)$, $\mathcal{J}(\overline\alpha_i,t)=1-\lambda\eta(t)M_0.$ The representation formula (\ref{eq:finalsolu22}) and (\ref{eq:bef22}) then imply
\begin{equation}
\label{eq:maxbef}
\begin{split}
u_x(\gamma(\overline\alpha_i,t),t)\geq u_x(\gamma(\alpha,t),t)
\end{split}
\end{equation} 
for $0\leq\eta(t)<\eta_*$ and $\alpha\in[0,1]$. Finally, one can easily see from (\ref{eq:finalsolu22}) that, as long as the solution exists,  
\begin{equation}
\label{eq:onetoone}
\begin{split}
u_x(\gamma(\alpha_1,t),t)=u_x(\gamma(\alpha_2,t),t)\,\,\,\,\,\Leftrightarrow\,\,\,\,\,u_0^\prime(\alpha_1)=u_0^\prime(\alpha_2)
\end{split}
\end{equation}
for all $\alpha_1, \alpha_2\in[0,1].$ Then, (\ref{eq:maxp})i) follows by using definition (\ref{eq:max}) and (\ref{eq:maxbef}). Likewise, $u_0^{\prime}(\alpha)\geq m_0=u_0'(\underline\alpha_j),$ (\ref{eq:finalsolu22}) and (\ref{eq:barg}) imply (\ref{eq:maxp})ii). Similarly for $\lambda<0,$ (\ref{eq:barg}) holds with $\eta_*=\frac{1}{\lambda m_0}>0$ instead. Both (\ref{eq:maxp})i), ii) then follow as above. $\square$

\end{document}